\documentclass[11pt, reqno]{amsart}

\usepackage[top=3.75cm, bottom=3cm, left=3cm, right=3cm]{geometry}
\usepackage{amsmath}
\usepackage{amsthm}
\usepackage{amsfonts}
\usepackage{amssymb}
\usepackage{bm}
\usepackage{braket}
\usepackage{enumitem}
\usepackage{mathtools}
\usepackage{tikz-cd}
\usepackage{tikz}
\usepackage{graphicx}
\usepackage{scalerel}
\usepackage{comment}

\numberwithin{equation}{section}

\setlist[itemize]{leftmargin=*, itemsep={2pt}}
\setlist[1]{labelindent=\parindent}
\setlist[1]{labelsep=0.5em}
\setlist[enumerate,1]{label={\upshape (\roman*)}, ref={\upshape (\roman*)}}

\usepackage{hyperref} 
\hypersetup{
  colorlinks = true,
  linkcolor = reference,
  urlcolor = citation,
  citecolor = citation,
  linkbordercolor = {white}}
  
\definecolor{citation}{rgb}{0,.40,.80}
\definecolor{reference}{rgb}{.80,0,.40}

\theoremstyle{plain}
\newtheorem{thm}{Theorem}[section]
\newtheorem{lemma}[thm]{Lemma}
\newtheorem{prop}[thm]{Proposition}
\newtheorem{cor}[thm]{Corollary}
\newtheorem{conjecture}[thm]{Conjecture}

\theoremstyle{definition}
\newtheorem{remark}[thm]{Remark}

\newcommand{\Db}{\mathrm{D^b}}

\newcommand{\st}{\mid}
\newcommand{\llangle}{\left \langle}
\newcommand{\rrangle}{\right \rangle}
\newcommand{\bZ}{\mathbf{Z}}
\newcommand{\cA}{\mathcal{A}}
\newcommand{\cB}{\mathcal{B}}

\newcommand{\perf}{\mathrm{perf}}
\newcommand{\tX}{\widetilde{X}}
\newcommand{\tKu}{\widetilde{\Ku}}
\newcommand{\Ku}{\mathcal{K}u}
\newcommand{\rH}{\mathrm{H}}
\newcommand\doubleplus{{\ensuremath{\mathbin{+\mkern-10mu+}}}}

\newcommand{\cHom}{\mathcal{H}\!{\it om}}

\DeclareMathOperator{\Ext}{Ext}
\DeclareMathOperator{\Sym}{Sym}
\DeclarePairedDelimiter{\abs}{\lvert}{\rvert}
\DeclareMathOperator{\Pic}{Pic}

\DeclareMathOperator{\Hom}{Hom}

\newcommand{\sO}{\mathcal{O}}
\newcommand{\PP}{\mathbf{P}}

\makeatletter
\def\imod#1{\allowbreak\mkern5mu({\operator@font mod}\,\,#1)}
\makeatother

\makeatletter
\newcommand*{\coloneqq}{\mathrel{\rlap{%
           \raisebox{0.3ex}{$\m@th\cdot$}}%
           \raisebox{-0.3ex}{$\m@th\cdot$}}%
           =}
\newcommand{\eqqcolon}{=%
           \mathrel{\rlap{%
           \raisebox{0.3ex}{$\m@th\cdot$}}%
           \raisebox{-0.3ex}{$\m@th\cdot$}}}
\makeatother

\makeatletter
\newcommand*\rel@kern[1]{\kern#1\dimexpr\macc@kerna}
\newcommand*\widebar[1]{%
  \begingroup
  \def\mathaccent##1##2{%
    \rel@kern{0.8}%
    \overline{\rel@kern{-0.8}\macc@nucleus\rel@kern{0.2}}%
    \rel@kern{-0.2}%
  }%
  \macc@depth\@ne
  \let\math@bgroup\@empty \let\math@egroup\macc@set@skewchar
  \mathsurround\z@ \frozen@everymath{\mathgroup\macc@group\relax}%
  \macc@set@skewchar\relax
  \let\mathaccentV\macc@nested@a
  \macc@nested@a\relax111{#1}%
  \endgroup
}
\makeatother

\newcommand*\circled[1]{\tikz[baseline=(char.base)]{
  \node[shape=circle,draw,inner sep=1pt,thick,scale=0.75] (char) {\textbf{#1}};}}

\newcommand{\mutation}{
  \begin{tikzpicture}[scale=0.95, every node/.style={scale=0.95}]
\matrix(m)[matrix of math nodes, column sep=0.5em, row sep=0.3em, ampersand replacement=\&]{
  \circled{0} \& (0,-1)      \& (0,0)     \& (0,1)     \& (0,2)     \& (1,0)     \& (1,1)     \& (1,2)     \& (1,3)     \\
  \circled{1} \& (-1,0)      \& (-1,1)    \& (0,-1)    \& (0,0)     \& (0,1)     \& (0,2)     \& (1,0)     \& (1,1)     \\
  \circled{2} \& (-1,0)      \& (0,-1)    \& (-1,1)    \& (0,0)     \& (0,1)     \& (1,0)     \& (0,2)     \& (1,1)     \\
  \circled{3} \& (0,-1)      \& i_*(0,-1) \& (0,0)     \& i_*(0,0)  \& i_*(1,0)  \& (0,1)     \& i_*(1,1)  \& (0,2)     \\
  \circled{4} \& i_*(0,-1)   \& (0,0)     \& i_*(0,0)  \& i_*(1,0)  \& (0,1)     \& i_*(1,1)  \& (0,2)     \& (2,1)     \\
  \circled{5} \& i_*(0,-1)   \& i_*(0,0)  \& (-1,1)    \& i_*(1,0)  \& (0,1)     \& (1,1)     \& i_*(1,1)  \& (2,1)     \\
  \circled{6} \& i_*(0,-1)   \& i_*(0,0)  \& i_*(1,0)  \& (-1,1)    \& (0,1)     \& (1,1)     \& (2,1)     \& i_*(1,1)  \\
  \circled{7} \& i_*(-1,-1)  \& i_*(0,-1) \& i_*(0,0)  \& i_*(1,0)  \& (-1,1)    \& (0,1)     \& (1,1)     \& (2,1)     \\
  \circled{8} \& i_*(0,-2)   \& i_*(1,-2) \& i_*(1,-1) \& i_*(2,-1) \& (0,0)     \& (1,0)     \& (2,0)     \& (3,0)     \\
};

\draw[->, shorten <=2mm] (m-1-8.south west) -- (m-1-9.south east);
\draw[->, shorten <=2mm] (m-4-2.south east) -- (m-4-2.south west);
\draw[->, shorten <=2mm] (m-7-9.south west) -- (m-7-9.south east);
\path[<->]
  (m-2-4) edge (m-2-3)
  (m-2-8) edge (m-2-7)
  (m-6-5) edge (m-6-4)
  (m-6-8) edge (m-6-9);

\path[<-]
  (m-3-3) edge (m-3-2)
  (m-3-5) edge (m-3-4);
\path[->]
  (m-3-7) edge (m-3-6)
  (m-3-9) edge (m-3-8);
\path[<-]
  (m-5-7) edge (m-5-8)
  (m-5-4) edge (m-5-3);
\end{tikzpicture}
}

\title{Derived categories of quartic double fivefolds}

\author{Raymond Cheng}
\address{Institute of Algebraic Geometry \\
  Leibniz University Hannover \\
  30167 Hannover \\
  Germany
}
\email{cheng@math.uni-hannover.de}
\urladdr{https://chngr.github.io/}

\author{Alexander Perry}
\address{Department of Mathematics \\ University of Michigan \\ Ann Arbor, MI 48109}
\email{arper@umich.edu}
\urladdr{http://www-personal.umich.edu/~arper/}

\author{Xiaolei Zhao}
\address{Department of Mathematics \\
University of California \\
Santa Barbara, CA 93106, USA
}
\email{xlzhao@math.ucsb.edu}
\urladdr{https://sites.google.com/site/xiaoleizhaoswebsite/}

\begin{document}

\begin{abstract}
We construct singular quartic double fivefolds whose Kuznetsov component admits
a crepant categorical resolution of singularities by a twisted Calabi--Yau
threefold. We also construct rational specializations of these fivefolds where
such a resolution exists without a twist. This confirms an instance of a
higher-dimensional version of Kuznetsov's rationality conjecture, and of a
noncommutative version of Reid's fantasy on the connectedness of the moduli of
Calabi--Yau threefolds.
\end{abstract}
\maketitle

\section{Introduction}
We work over an algebraically closed field $k$ of characteristic $0$.
Let $X$ be a smooth prime Fano variety of index $r$, i.e. $\Pic(X) \cong \bZ$ is generated by an ample line bundle $\sO_X(1)$ such that $\omega_X \cong \sO_X(-r)$.
Then the bounded derived category of coherent sheaves on $X$ admits a semiorthogonal
decomposition
\begin{equation}
\label{DbY}
\Db(X) = \llangle
\Ku(X), \sO_X, \sO_X(1), \dots, \sO_X(r-1)
\rrangle
\end{equation}
where $\Ku(X) \subset \Db(X)$ is the \emph{Kuznetsov component}, defined explicitly as the full subcategory
\begin{equation*}
\Ku(X) = \set{ F \in \Db(X) \st \Ext^\bullet(\sO_X(i), F) = 0 \text{ for } 0 \leq i \leq r-1 } .
\end{equation*}
The category $\Ku(X)$ should be thought of as the ``interesting part'' of $\Db(X)$ obtained as the orthogonal to some tautological pieces coming from the polarization of $X$.

\begin{remark}
In some cases, $\Ku(X)$
can be refined by taking the orthogonal to some additional tautological objects on $X$
(see \cite{Kuz:survey} for examples), but in this paper we will only be concerned with $\Ku(X)$ as defined above.
\end{remark}

Kuznetsov components have recently been very influential in algebraic geometry, due in part to their close connections to birational geometry.
The most famous example is when
$X \subset \PP^5$ is a cubic fourfold, in which case
Kuznetsov \cite{Kuz:cubic} conjectured that $X$ is rational if and only if
$\Ku(X)$ is equivalent to the derived category of a K3 surface.
This has been verified for all known rational cubic fourfolds, and there is now
a precise Hodge-theoretic characterization of when $\Ku(X)$ is equivalent to the derived category of a K3 surface \cite{addington-thomas, stability-families}, but in general 
the conjecture remains tantalizingly open.

The heuristics behind Kuznetsov's rationality conjecture suggest more generally that if $\Ku(X)$ is a Calabi--Yau category of dimension $\dim(X)-2$ and $X$ is rational,
then $\Ku(X)$ is equivalent to the derived category of a smooth projective Calabi--Yau variety
(see \cite{kuznetsov-rationality}).
Most work on this problem has been confined to the case $\dim(X)  = 4$.
The purpose of this paper is to investigate an interesting $5$-dimensional example.

Namely, we take $X$ to be a quartic double fivefold, i.e.\ a double cover $X \to \PP^5$ branched along a quartic hypersurface, 
whose Kuznetsov component is defined by the semiorthogonal decomposition 
\begin{equation}
\label{KuX}
\Db(X) = \llangle \Ku(X), \sO_X, \sO_X(1), \sO_X(2), \sO_X(3) \rrangle.
\end{equation}
By \cite[Corollary 4.6]{kuznetsov-CY} the category $\Ku(X)$ is Calabi--Yau of dimension $3$.  
Therefore, Kuznetsov's philosophy suggests the following.

\begin{conjecture}
\label{conjecture-DbX}
If $X$ is a smooth quartic double fivefold which is rational, then there exists an equivalence $\Ku(X) \simeq \Db(W)$ for a smooth projective Calabi--Yau threefold $W$.
\end{conjecture}

This conjecture appears to be quite difficult.
In fact, it follows from a Hochschild homology computation that $\Ku(X)$ is not equivalent to the derived category of any smooth projective variety (see \cite[Lemma 6.9]{HH}), 
so the conjecture is equivalent to the irrationality of every smooth quartic double fivefold.
This remains open despite many recent breakthroughs on the rationality problem. 
In fact, at the moment no quartic double fivefold is known to be irrational, and with the current techniques the best one could hope to prove is irrationality in the very general case.

In this paper, we instead investigate the situation for certain singular
quartic double fivefolds $X$, which are more tractable. 
In this situation, the subcategory $\Ku(X) \subset \Db(X)$ may be defined by the same semiorthogonal 
decomposition~\eqref{KuX} as in the smooth case. 
Our first main result is as follows. 

\begin{thm}
\label{main-theorem}
Let \(X \to \PP^5\) be a double cover branched along a quartic hypersurface
\(Y \subset \PP^5\) which is singular along a line \(L \subset \PP^5\).
For general such $Y$, there is a crepant categorical resolution of
singularities of \(\Ku(X)\) by \(\Db(W^+,\cA^+)\), where \(W^+\) is a proper
\(3\)-dimensional Calabi--Yau algebraic space which is not projective, and
\(\cA^+\) is an Azumaya algebra on \(W^+\) whose Brauer class is nontrivial.
\end{thm}

\begin{remark}
The notion of a crepant categorical resolution\footnote{To be precise, we use the term ``crepant categorical resolution'' for what is called a ``weakly crepant categorical resolution'' in \cite{Kuz:cat-res}.}
is due to Kuznetsov \cite{Kuz:cat-res} and abstracts the properties of the derived category of a crepant resolution of singularities of a variety with rational singularities.
\end{remark}

The proof of Theorem~\ref{main-theorem} is based on a study of the derived category of a natural resolution of singularities \( \tX \to X\).
Roughly, we define a Kuznetsov component $\tKu(X) \subset \Db(\tX)$
which is a crepant categorical resolution of $\Ku(X)$, and use the quadric bundle structure on $\tX$ induced by linear projection from $L$ to identify $\tKu(X)$ with the derived category of an associated pair $(W^+, \cA^+)$.

We were originally motivated by a version of Conjecture~\ref{conjecture-DbX} allowing singularities, which for rational $X$ asks for a crepant resolution of $\Ku(X)$ by a Calabi--Yau threefold, instead of a derived equivalence with one (which is impossible when $X$ is singular). 
However, we do not know whether $X$ in Theorem~\ref{main-theorem} is rational, and there are some indications that it is not. 
First, the nonvanishing of the Brauer class of \(\mathcal{A}^{+}\) is equivalent to the nonexistence of a rational section for the quadric bundle $\tX$ mentioned above (Lemma~\ref{general-rationality}). 
Second, the construction produces a \emph{non-projective} Calabi--Yau threefold $W^{+}$, in contrast to what is expected from Conjecture~\ref{conjecture-DbX}. 


To produce a sharper example, we specialize further to a situation where the quadric bundle \(\widetilde{X}\) 
admits a section. 
Our second main result is as follows. 

\begin{thm}
\label{main-theorem-special}
Let \(X \to \PP^5\) be a double cover branched along a quartic hypersurface
\(Y \subset \PP^5\). 
Assume that $Y$ is singular along a line $L \subset \PP^5$ and that there exists 
a \(3\)-plane \(P \subset \PP^5\) complementary to \(L\) which is tangent to
\(Y\) along a smooth quadric surface.
\begin{enumerate}
\item\label{main-theorem-special.resolution}
For general such \(Y\), there is a crepant categorical resolution of
singularities of \(\Ku(X)\) by \(\Db(W^\doubleplus)\), where \(W^\doubleplus\)
is a projective Calabi--Yau threefold.
\item\label{main-theorem-special.rational}
In the situation of \ref{main-theorem-special.resolution}, \(X\) is rational.
\end{enumerate}
\end{thm}

Theorems \ref{main-theorem} and \ref{main-theorem-special} give a procedure
for connecting the CY3 category $\Ku(X)$ of a smooth
$X$ first to a twisted geometric Calabi--Yau threefold \((W^+,\mathcal{A}^+)\),
and then to a geometric Calabi--Yau threefold \(W^\doubleplus\), where
each step proceeds by degenerating the CY3 category and then 
taking a crepant resolution; 
we expect this to be useful for studying the category $\Ku(X)$ and its moduli spaces of objects by deformation from the geometric case. 
There is a classical geometric version of this
procedure, the simplest example being a conifold transition, where a
Calabi--Yau threefold is degenerated and then crepantly resolved to obtain
another; such constructions have been widely studied in support of ``Reid's
fantasy'' \cite{reid} that all Calabi--Yau threefolds can be connected in this
way. Theorem~\ref{main-theorem} can similarly be regarded as evidence for the
noncommutative version of Reid's fantasy raised in~\cite{categorical-cones}.

Let us also note that Theorem~\ref{main-theorem-special}
provides an analog for quartic double fivefolds of
\cite[Theorem~5.2]{Kuz:cubic}, which gives a crepant resolution of the Kuznetsov component of a nodal (necessarily rational) cubic fourfold by the derived category of a K3 surface.

\begin{remark}[Higher dimensions] 
Quartic double fivefolds are the first in a series of higher-dimensional examples with similar properties. 
Namely, if $X \to \PP^{4m+1}$ is a double cover branched along a quartic hypersurface, then by \cite[Corollary 4.6]{kuznetsov-CY} the category $\Ku(X)$ is Calabi--Yau of dimension $2m+1$. 
We expect our arguments to also be useful for proving analogs of Theorems~\ref{main-theorem} and~\ref{main-theorem-special} when $m > 1$. 
\end{remark} 

\medskip
\noindent\textbf{Outline.}
Some basic facts about the geometry of quartic double fivefolds $X$ arising
from quartics singular along a line are worked out in~\S\ref{S:quartics}, after
which Theorems \ref{main-theorem} and \ref{main-theorem-special} are proven in
\S\ref{section-general} and \S\ref{section-special}, respectively.

\medskip
\noindent\textbf{Conventions.}
All functors are derived. In particular, for a morphism $f \colon X \to Y$ we
write $f^*$ and $f_*$ for the derived pullback and pushforward, and for $E, F
\in \Db(X)$ we write $E \otimes F$ for their derived tensor product.

\medskip
\noindent\textbf{Acknowledgements.}
This project originated at the workshop on ``Rationality problems in algebraic
geometry'' held at the American Institute of Mathematics from July 29, 2019, to
August 2, 2019. We would like to thank the Institute and our fellow
participants for a wonderful workshop. 
We are also grateful to Huachen Chen and Tony Feng for their collaboration in the early stages of this project, 
to Asher Auel and Stefan Schreieder for helpful discussions, 
and to Evgeny Shinder for suggesting the proof of the nontriviality of the Brauer class of $\cA^{+}$ in Theorem~\ref{main-theorem}. 
We also thank the referees for useful comments and suggestions.

This work was
completed in part while RC and XZ were at the Junior Trimester Program in
Algebraic Geometry at the Hausdorff Research Institute for Mathematics during
the autumn of 2023, funded by the Deutsche Forschungsgemeinschaft (DFG, German
Research Foundation) under Germany's Excellence Strategy - EXC-2047/1 - 390685813. 

During the preparation of this paper, 
RC was partially supported by a Research Fellowship from the Alexander von Humboldt Stiftung;
AP was partially supported by NSF grants DMS-2112747, DMS-2052750, and DMS-2143271, and a Sloan Research Fellowship; and
XZ was partially supported by Simons Collaborative Grant 636187, and NSF grants DMS-2101789 and DMS-2052665.


\section{Quartic double fivefolds singular along a line}\label{S:quartics}
In this section we analyze the geometry of a quartic double fivefold singular
along a line. Throughout, \(V\) is a \(6\)-dimensional vector space and
\(\PP V \cong \PP^5\) is the associated projective space of lines. Let
\(U \subset V\) be a \(2\)-dimensional subspace corresponding to a projective
line \(L \coloneqq \PP U\). Set \(\widebar{V} \coloneqq V/U\), choose a
splitting \(V \cong U \oplus \widebar{V}\), 
let \(P   \subset \PP V\) be the
corresponding projective \(3\)-space complementary to \(L\), 
and fix a smooth quadric surface \(S \subset \PP\widebar{V}\).

\subsection{Linear projection from the line}
Projection away from \(L\) induces a rational map
\(\PP V \dashrightarrow \PP\widebar{V}\), which is resolved on the
blow up \(b \colon \widetilde{\PP}V \to \PP V\) along \(L\).
Write \(E \hookrightarrow \widetilde{\PP}V\) for the exceptional divisor.
This data fits into a diagram
\[
\begin{tikzcd}
  E \rar[hook] \dar["b_E"] & \widetilde{\PP} V \dar["b"] \rar["a"'] & \PP\widebar{V} \\
  L \rar[hook] & \PP V.
\end{tikzcd}
\]
Let \(H\) be the hyperplane class on \(\PP V\) and
\(h\) the hyperplane class on \(\PP\widebar{V}\).
Some useful standard facts about the situation are as follows:

\begin{lemma}
\label{bl-facts}
Let \(\mathcal{F} \coloneqq (a_*\sO_{\widetilde{\PP}V}(H))^\vee\). Then
\begin{enumerate}
\item\label{bl-facts.F}
\(\mathcal{F} \cong \sO_{\PP\widebar{V}} \otimes U \oplus \sO_{\PP\widebar{V}}(-h)\);
\item\label{bl-facts.P}
\(a \colon \widetilde{\PP}V \to \PP\widebar{V}\) is isomorphic to
  the projective bundle \(\PP\mathcal{F} \to \PP\widebar{V}\);
\item \label{bl-facts.h}
  \(h = H - E\) in \(\Pic(\widetilde{\PP}V)\); and
\item \label{bl-facts.K}
  the canonical divisor of \(\widetilde\PP V\) is \(K_{\widetilde\PP V} = -6H + 3E = -3H - 3h\).
\end{enumerate}
\end{lemma}

\subsection{Quartic fourfolds}\label{quartic}
Within the complete linear system \(\abs{4H}\) of quartic fourfolds in
\(\PP V\), consider the linear systems
\begin{align*}
\mathfrak{a}
& \coloneqq \Set{Y \in \abs{4H} : Y\;\text{is singular along}\; L}\;\text{and} \\
\mathfrak{b}
& \coloneqq \Set{Y \in \mathfrak{a} : Y \;\text{is everywhere tangent to}\; P\;\text{along}\;S},
\end{align*}
of codimensions \(21\) and \(56\), respectively, consisting of quartics that
are singular along the line \(L\), and those which are furthermore tangent to
the \(3\)-space \(P\) along the quadric \(S\). Let \(Y \hookrightarrow \PP V\)
be a member of \(\mathfrak{a}\), and let
\(\widetilde{Y} \hookrightarrow \widetilde{\PP}V\) be the strict transform of
\(Y\) along the blow up \(b\). Since \(L\) has multiplicity at least \(2\) in \(Y\),
\(\widetilde{Y}\) is a member of the complete linear system
\(\abs{4H - 2E} = \abs{2H + 2h}\) on \(\widetilde{\PP}V\). Moreover, the
projection \(\widetilde{Y} \to \PP\widebar{V}\) exhibits \(\widetilde{Y}\) as a conic
bundle in \(\PP\mathcal{F} \to \PP\widebar{V}\) corresponding to a section
\(\theta \colon \sO_{\PP\widebar{V}} \to \Sym^2(\mathcal{F}^\vee) \otimes \sO_{\PP\widebar{V}}(2h)\).
Every such conic bundle arises in this way, and the ones corresponding to
\(Y \in \mathfrak{b}\) are those whose \(\sO_{\PP\widebar{V}}(4h)\) component
of \(\theta\) is a scalar multiple of an equation for \(2S \subset P \cong \PP\widebar{V}\).

The following describes the singularities of
\(\widetilde{Y}\) when it arises from a general member of \(\mathfrak{a}\) or
\(\mathfrak{b}\). In the latter case, the surface \(S \hookrightarrow Y\) may
be viewed as a subscheme of \(\widetilde{Y}\) since \(S\) is disjoint from
\(L\). Recall that an isolated singularity is called an \emph{ordinary double
point} or a \emph{node} if it is a hypersurface singularity defined by a
homogeneous quadratic equation of full rank.

\begin{lemma}
\label{quartic-singularities}
The strict transform \(\widetilde{Y}\) along the blow up
\(b \colon \widetilde{\PP} V \to \PP V\)
\begin{enumerate}
\item\label{quartic-singularities.L}
is smooth for general \(Y \in \mathfrak{a}\), and
\item\label{quartic-singularities.P}
has only \(18\) nodes along \(S \hookrightarrow \widetilde{Y}\)
for general \(Y \in \mathfrak{b}\).
\end{enumerate}
\end{lemma}

\begin{proof}
For \ref{quartic-singularities.L}, if \(Y\) is a general member of
\(\mathfrak{a}\), then \(\widetilde Y\) is a general member of \(|2H+2h|\);
since this latter linear system is base point free, Bertini's
theorem implies \(\widetilde{Y}\) is smooth. For \ref{quartic-singularities.P},
choose linear forms \(y_1\) and \(y_2\) on \(\PP V\) which together cut out the
\(3\)-plane \(P\). Then a quartic \(Y\) in the linear system \(\mathfrak{b}\)
is defined by an equation
of the form
\[
Y =
\mathrm{V}(
  \beta_{11} y_1^2 + \beta_{12} y_1y_2 + \beta_{22} y_2^2 +
  \alpha_1 y_1 + \alpha_2 y_2 + q^2
)
\]
where \(\beta_{11}, \beta_{12}, \beta_{22} \in \mathrm{H}^0(P, \sO_P(2))\),
\(\alpha_1, \alpha_2 \in \mathrm{H}^0(P,\sO_P(3))\), and \(q\) is an equation
of \(S\) in \(P\); here, functions on \(P\) are viewed as functions on
\(\PP V\) via the splitting \(V \cong U \oplus \widebar{V}\). The base points
of \(\mathfrak{b}\) are contained in the line \(L\) and the quadric surface
\(S\), and since \(y_1\) and \(y_2\) span the space of linear functions on
\(L\), the base points of the strict transform of \(\mathfrak{b}\) on
\(\widetilde{\PP} V\) are contained in \(S\). Therefore the singularities of
\(\widetilde{Y}\) for general \(Y \in \mathfrak{b}\) are contained in \(S\),
and are those points where the above equation vanishes to order at least \(2\),
and these are the points where both cubics \(\alpha_1\) and \(\alpha_2\)
vanish. It remains to observe that when these two cubics intersect the quadric
\(q\) in the projective \(3\)-space \(P\) at \(18\) reduced points, the
corresponding singularities on \(\widetilde{Y}\) are nodes: indeed, this means
that the linear terms of \(y_1\), \(y_2\), \(\alpha_1\), \(\alpha_2\), and
\(q\) in formal local coordinates are linearly independent, and so the tangent
cone therein is a full rank quadric.
\end{proof}

\subsection{Double quartic fivefolds}\label{double-quartic-fivefolds-setup}
Let \(f \colon X \to \PP V\) and
\(\widetilde{f} \colon \widetilde{X} \to \widetilde{\PP}V\) be the double
covers branched along \(Y\) and \(\widetilde{Y}\). Lemma
\ref{quartic-singularities} shows that \(\widetilde{X}\) is smooth for
general \(Y \in \mathfrak{a}\), and has only \(18\) nodes along the preimage
of \(S\) for general \(Y \in \mathfrak{b}\). In both cases,
the canonical morphism \(b_X \colon \widetilde{X} \to X\) resolves the
singularities along \(L\). Writing \(Z \coloneqq b_X^{-1}(L)\) for the
exceptional divisor and
\(\pi \coloneqq a \circ \widetilde{f} \colon \widetilde{X} \to \PP\widebar{V}\),
there is a commutative diagram
\[
  \begin{tikzcd}
    Z \rar["i"', hook] \dar["b_Z"] & \widetilde{X} \dar["b_X"] \rar["\pi"'] & \PP\widebar{V} \\
    L \rar[hook] & X.
  \end{tikzcd}
\]
%
The essential point is that \(\widetilde{X}\) is a quadric surface bundle
over \(\PP\widebar{V}\). To simplify notation, write \(H\) and \(h\) for the
pullback under \(\widetilde{X} \to \widetilde{\PP} V\) of the corresponding
divisor classes. We have:

\begin{lemma}
\label{quadric-bundle}
Let \(\mathcal{E} \coloneqq (\pi_*\sO_{\widetilde X}(H))^\vee\). Then
\begin{enumerate}
\item\label{quadric-bundle.E}
\(\mathcal{E}
\cong \mathcal{F} \oplus \sO_{\PP\widebar{V}}(h)
\cong \sO_{\PP\widebar{V}}(-h) \oplus \sO_{\PP\widebar{V}} \oplus \sO_{\PP\widebar{V}} \oplus \sO_{\PP\widebar{V}}(h)\);
\item\label{quadric-bundle.class}
\(\widetilde X\) embeds into \(\PP\mathcal{E}\) as a hypersurface of class
\(2H + 2h\); and
\item\label{quadric-bundle.K}
the canonical class of \(\widetilde X\) is \(K_{\widetilde X} = -2H - 2h\).
\end{enumerate}
\end{lemma}

\begin{proof}
For \ref{quadric-bundle.E}, compute using the projection formula and
Lemma~\ref{bl-facts}\ref{bl-facts.h}:
\[
  \pi_*\sO_{\widetilde X}(H)
  = a_*\widetilde f_*\widetilde f^* \sO_{\widetilde{\PP}V}(H)
  = a_*\big(\sO_{\widetilde{\PP}V}(H)
    \oplus \sO_{\widetilde{\PP}V}(-H+E)\big)
  = \mathcal{F}^\vee \oplus \sO_{\PP\widebar{V}}(-h).
\]

For \ref{quadric-bundle.class}, note that
\(\widetilde{X} \to \widetilde{\PP} V\) is a double cover of \(\PP^2\) branched
along a conic Zariski-locally over \(\PP\widebar{V}\), so the canonical map
\(\widetilde{X} \to \PP\mathcal{E}\) is an embedding as a relative quadric surface.
Since the relative hyperplane class of \(\PP\mathcal{E} \to \PP\widebar{V}\) is
the pullback of \(H\), and since \(\Pic(\PP\mathcal{E}) \cong \mathbf{Z} H
\oplus \mathbf{Z} h\), the class of \(\widetilde X\) in \(\PP\mathcal{E}\) is
\(2H + nh\) for some integer \(n\). To determine \(n\), consider the sequence
\[
0 \to
\sO_{\PP\mathcal{E}}(-nh) \to
\sO_{\PP\mathcal{E}}(2H) \to
\sO_{\widetilde X}(2H) \to
0.
\]
Pushing forward to \(\PP\widebar{V}\) and computing as in \ref{quadric-bundle.E}
gives
\[
0 \to
\sO_{\PP\widebar{V}}(-nh) \to
\Sym^2(\mathcal{E}^\vee) \to
\Sym^2(\mathcal{F}^\vee) \oplus \mathcal{F}^{\vee}(-h) \to
0.
\]
Since
\(\Sym^2(\mathcal{E}^\vee) \cong \Sym^2(\mathcal{F}^\vee) \oplus \mathcal{F}^{\vee}(-h) \oplus \sO_{\PP\widebar{V}}(-2h)\),
taking Chern classes shows \(n = 2\).

For \ref{quadric-bundle.K}, use Lemma \ref{bl-facts}\ref{bl-facts.K} together
with the description of \(\widetilde X\) as a double cover of \(\widetilde\PP
V\) branched along a divisor of type \(2H + 2h\):
\[
K_{\widetilde X}
= \widetilde f^* K_{\widetilde \PP V} + \frac{1}{2}(2H + 2h)
= (-3H - 3h) + (H + h) = -2H - 2h.
\qedhere
\]
\end{proof}

The following statement collects some basic facts about the geometry of
\(Z\):

\begin{lemma}
\label{exceptional-divisor-Z}
The exceptional divisor \(Z \subset \widetilde{X}\)
\begin{enumerate}
\item\label{exceptional-divisor-Z.double-cover}
is a double cover of \(E \cong L \times \PP\widebar{V}\) branched along the
divisor \(E \cap \widetilde{Y}\) of class \(2H+2h\);
\item\label{exceptional-divisor-Z.divisor-class}
has divisor class \(Z = H - h\) on \(\widetilde{X}\);
\item\label{exceptional-divisor-Z.quadric-fibration}
is a quadric threefold fibration over \(L\) via \(b_Z \colon Z \to L\); and
\item\label{exceptional-divisor-Z.smooth}
is smooth for \(Y \in \mathfrak{a}\) general.
\end{enumerate}
\end{lemma}

\begin{proof}
Item \ref{exceptional-divisor-Z.double-cover} is the fact that \(Z\) is the
restriction of the double covering
\(\widetilde{f} \colon \widetilde{X} \to \widetilde{\PP}V\) over the
exceptional divisor \(E \subset \widetilde{\PP}V\). Combined with Lemma
\ref{bl-facts}\ref{bl-facts.h}, this gives
\ref{exceptional-divisor-Z.divisor-class}.

To see \ref{exceptional-divisor-Z.quadric-fibration}, observe that
\(b_Z \colon Z \to L\) factors through \(E\), and that each fiber of is a
double covering of \(\PP^3\) branched along a quadric surface by
\ref{exceptional-divisor-Z.double-cover}, and so, as in Lemma
\ref{quadric-bundle}, is a quadric.

Finally, \ref{exceptional-divisor-Z.smooth} follows from Bertini's theorem,
which shows that the branch locus \(E \cap \widetilde{Y}\), whence the double
cover \(Z \to E\), is smooth for general \(Y \in \mathfrak{a}\), as the linear
system \(2H + 2h\) on \(E\) is base point free.
\end{proof}

Let \(D \hookrightarrow \PP\widebar{V}\) be the discriminant locus of the
quadric surface bundle \(\pi \colon \widetilde{X} \to \PP\widebar{V}\): as
usual, this is the subscheme over which fibers are singular quadrics, or
equivalently, the locus over which the associated bilinear form has corank at
least \(1\). When \(\pi\) is generically smooth, \(D\) is a surface
and, by \cite[Proposition 1.2.5]{ABB}, its singular locus consists of the
subscheme \(D_0\) over which the bilinear form furthermore has corank at least
\(2\) together with the image of the singular locus of \(\widetilde{X}\). In
the situation at hand, \(D\) is as follows:

\begin{lemma}
\label{discr}
For general quartics \(Y\) in either \(\mathfrak{a}\) or \(\mathfrak{b}\),
the discriminant locus \(D\) of \(\pi\) is a surface of degree \(8\) with
singular locus consisting of
\begin{enumerate}
\item\label{discr.sings}
only \(72\) nodes along \(D_0\) for general \(Y \in \mathfrak{a}\); and
\item\label{discr.sings-P}
only \(18\) additional nodes corresponding to those of \(\widetilde{X}\) for
general \(Y \in \mathfrak{b}\).
\end{enumerate}
\end{lemma}

\begin{proof}
The fibers of \(\pi \colon \widetilde{X} \to \PP\widebar{V}\) are double planes
branched over the conic fibers of \(\widetilde{Y} \to \PP\widebar{V}\), so
\(D\) is also the discriminant locus of the latter conic bundle. Since
\(\widetilde{Y}\) is defined by a section of
\(\Sym^2(\mathcal{F}^\vee) \otimes \mathcal{O}_{\PP\widebar{V}}(2h)\), writing
\(c_i \coloneqq c_i(\mathcal{F}^\vee \otimes \sO_{\PP\widebar{V}}(h))\),
\cite[Theorem 10]{HT:deg} or \cite[Example 14.4.11]{Fulton} apply to give
identities
\[
  [D] = 2c_1 \quad\text{and}\quad
  [D_0] = 4 \det\begin{pmatrix} c_2 & c_3 \\ c_0 & c_1 \end{pmatrix}
\]
in the Chow ring of \(\PP\widebar{V}\) whenever \(D\) and \(D_0\) are
of expected dimensions \(2\) and \(0\), respectively. These have degrees
\(8\) and \(72\), respectively, using
\[
  c(\mathcal{F}^\vee \otimes \sO_{\PP\widebar{V}}(h))
  = (1 + h)^2(1 + 2h)
  = 1 + 4h + 5h^2 + 2h^3.
\]
For general \(Y\) in either \(\mathfrak{a}\) or \(\mathfrak{b}\),
\(\widetilde{Y}\) is generically smooth and so \(D\) is a surface of degree
\(8\). Since the vector bundle
\(\Sym^2(\mathcal{F}^\vee) \otimes \sO_{\PP\widebar{V}}(2h)\) defining
\(\widetilde{Y}\) is globally generated, and since the vector space underlying
\(\mathfrak{a}\) is canonically identified with its space of sections,
a Bertini-type argument as in \cite[Lemma 4]{Barth:Quadrics} shows that the
singular locus of \(D\) consists of nodes supported on the \(0\)-dimensional
locus \(D_0\), giving \ref{discr.sings}.

Similarly, since \(\Sym^2(\mathcal{F}^\vee) \otimes \sO_{\PP\widebar{V}}(2h)\)
is generated away from \(S \hookrightarrow \PP\widebar{V}\) by its sections
corresponding to members of \(\mathfrak{b}\), the Bertini argument also shows
that, for general \(Y \in \mathfrak{b}\), \(D\) has only nodes along \(D_0\)
away from \(S\). Thus to prove \ref{discr.sings-P}, it remains to show that
\(D_0\) is disjoint from \(S\) and that \(D\) has only nodes along \(S\) for
general \(Y \in \mathfrak{b}\). For this, and for later use, note that for any
\(Y \in \mathfrak{b}\), a symmetric bilinear form defining the quadric
surface bundle \(\pi \colon \widetilde{X} \to \PP\widebar{V}\) may be written
with the notation of Lemma \ref{quartic-singularities} as
\begin{equation}\label{X.bilinear-form}
A \coloneqq \frac{1}{2}
\begin{pmatrix}
-2 & 0       & 0       & 0 \\
0  & 2\beta_{11} & \beta_{12}  & \alpha_1 \\
0  & \beta_{12}  & 2\beta_{22} & \alpha_2 \\
0  & \alpha_1     & \alpha_2     & 2q^2
\end{pmatrix} \colon
\mathcal{E} \to \mathcal{E}^\vee \otimes \sO_{\PP\widebar{V}}(2h)
\end{equation}
where the matrix is with respect to the decompositions 
\begin{align*}
\mathcal{E} & \cong \sO_{\PP\widebar{V}}(h) \oplus \sO_{\PP\widebar{V}} \oplus \sO_{\PP\widebar{V}} \oplus \sO_{\PP\widebar{V}}(-h) \\ 
\mathcal{E}^{\vee} \otimes \sO_{\PP\widebar{V}}(2h) & \cong \sO_{\PP\widebar{V}}(h) \oplus \sO_{\PP\widebar{V}}(2h) \oplus \sO_{\PP\widebar{V}}(2h) \oplus \sO_{\PP\widebar{V}}(3h). 
\end{align*}
Therefore an equation for \(D\) in \(\PP\widebar{V}\) is given by
\[
-4\det(A) =
\alpha_2^2\beta_{11} - \alpha_1\alpha_2 \beta_{12} + \alpha_1^2\beta_{22} - q^2(4\beta_{11}\beta_{22} - \beta_{12}^2).
\]
The singularities of \(D\) away from \(D_0\) lie in the
image of the singular locus of \(\widetilde{X}\), which by Lemma
\ref{quartic-singularities}, is the subscheme
\(\mathrm{V}(\alpha_1, \alpha_2, q)\). This is disjoint from \(D_0\) since
none of \(\beta_{11}\), \(\beta_{12}\), \(\beta_{22}\), nor
\(4\beta_{11}\beta_{22} - \beta_{12}^2\) vanish there for general
\(Y \in \mathfrak{b}\). This moreover implies that the tangent cone to \(D\) at
points therein is a quadric defined by products of linear terms of
\(\alpha_1\), \(\alpha_2\), and \(q\), and so, arguing as in
Lemma \ref{quartic-singularities}, \(D\) has only nodes precisely when the
hypersurfaces defined by \(\alpha_1\), \(\alpha_2\), and \(q\) intersect
transversally in \(\PP\widebar{V}\), yielding \ref{discr.sings-P}.
\end{proof}

Since \(\PP\widebar{V}\) is, of course, smooth, a node \(x \in \widetilde{X}\)
must be contained in the nonsmooth locus of the map \(\pi\). So
\(\pi(x) \in D\) and it is a singular point of \(D\). Since the proof of Lemma
\ref{discr}\ref{discr.sings-P} shows that, for general \(Y \in \mathfrak{b}\),
the singularities of \(D\) corresponding to those of \(\widetilde{X}\) lie away
from the corank \(2\) locus \(D_0\), this implies:

\begin{cor}
\label{quartics-nodes-are-cone-points}
For general \(Y \in \mathfrak{b}\), each node of \(\widetilde{X}\) is the
cone point of a corank \(1\) fiber of \(\pi\).
\end{cor}

\subsection{Sections of the singular quadric surface bundles}
\label{quartics-quadric-bundle-sections}
Consider the singular quadric surface bundles
\(\pi \colon \widetilde{X} \to \PP\widebar{V}\) arising from a general member
\(Y\) of the linear system \(\mathfrak{b}\). Since the \(3\)-plane \(P\)
intersects the branch locus \(Y\) doubly along the quadric surface
\(S\), its preimage along the double cover \(f \colon X \to \PP V\) is
reducible, and so is a union \(f^{-1}(P) = P^+ \cup P^-\) of two components,
each isomorphic to \(P\). As \(P\) is disjoint from the line \(L\), the
\(P^\pm\) may be identified as subschemes of \(\widetilde{X}\), providing two
sections
\[
\sigma^\pm \colon
  \PP\widebar{V} \xrightarrow{\sim}
  P^\pm \hookrightarrow
  \widetilde{X}
\]
to \(\pi\); in particular, this implies that \(\widetilde{X}\) is rational.
Moreover, upon examining the bilinear form \eqref{X.bilinear-form},
\(\sigma^\pm\) are seen to correspond to the line subbundles
\[
\mathcal{N}^\pm \coloneqq
\operatorname{image}\big(
  (\pm q, 0, 0, 1)^t \colon
  \sO_{\PP\widebar{V}}(-h) \hookrightarrow
  \mathcal{E}
\big).
\]
The basic fact about the geometry of these sections is:

\begin{lemma}
\label{quartics-sections-mostly-smooth}
Each section \(\sigma^\pm\) passes through every node of \(\widetilde{X}\), and
are otherwise contained in the smooth locus of \(\pi\).
\end{lemma}

\begin{proof}
Lemma \ref{quartic-singularities}\ref{quartic-singularities.P} implies that
the nodes of \(\widetilde{X}\) lie over \(Y \cap P\), so \(\sigma^\pm\) must
pass through all of them by construction. That \(\sigma^\pm\) are otherwise
contained in the smooth locus of \(\pi\) follows from noting that
\(\mathcal{N}^\pm \hookrightarrow \mathcal{E}\) does not intersect the kernel
of the bilinear form \eqref{X.bilinear-form} at points where
at least one of \(\alpha_1\), \(\alpha_2\), or \(q\) is nonvanishing.
\end{proof}


\section{General situation: twisted geometric component}
\label{section-general}
This section is concerned with the double quartic fivefolds that arise from
a general quartic fourfold singular along the line \(L\). So fix a general
member \(Y\) in the linear system \(\mathfrak{a}\), and continue with the
notation in \S\ref{S:quartics}.

\subsection{Crepant resolution of \(\Ku(X)\)}\label{S:D}
First we construct a crepant resolution of the Kuznetsov component of \(X\)
using the geometric resolution of singularities
\(b_X \colon \widetilde{X} \to X\). Recall that we define the Kuznetsov
component of a smooth prime Fano variety by the semiorthogonal
decomposition~\eqref{DbY}. This semiorthogonal decomposition still exists in
many situations when the Fano variety is singular, and in particular in our
setting:

\begin{lemma}
\label{sod-X}
There is a semiorthogonal decomposition
\[
  \Db(X) = \langle \Ku(X), \sO_X, \sO_X(H), \sO_X(2H), \sO_X(3H) \rangle,
\]
where $\Ku(X) \subset \Db(X)$ is the full subcategory defined by
\begin{equation*}
\Ku(X) = \set{ F \in \Db(X) \st \Hom^\bullet(\sO_X(iH), F) = 0 \text{ for } 0 \leq i \leq 3 } .
\end{equation*}
\end{lemma}

\begin{proof}
This is a special case of \cite[Lemma 5.1]{cyclic-covers}, but for convenience
we recall the proof in our situation. Since \(X\) is a double cover of
\(\PP V\) branched along a quartic,
\[
  \rH^\bullet(X,\sO_X(mH)) =
  \rH^\bullet(\PP V, \sO_{\PP V}(mH) \oplus \sO_{\PP V}((m - 2)H))
\]
for all integers \(m\).
Taking \(m = 0\) shows that \(\sO_X\) is an exceptional object; hence the same applies to all of the line bundles
\(\sO_X(nH)\).
Similarly, it is easy to see that the displayed cohomology group vanishes for $-3 \leq m \leq -1$,
so the objects $\sO_X, \sO_X(H), \sO_X(2H), \sO_X(3H)$ are semiorthogonal.
Since $\Ku(X)$ is by definition the right orthogonal to the admissible subcategory
generated by these objects, the result follows.
\end{proof}

To construct a resolution of $\Ku(X)$, we will need a suitable semiorthogonal
decomposition of the exceptional divisor $Z \subset \tX$. By
Lemma \ref{exceptional-divisor-Z}\ref{exceptional-divisor-Z.quadric-fibration},
\(b_Z \colon Z \to L\) is a quadric threefold fibration. Applying Serre duality
to the quadric bundle decomposition from \cite[Theorem~4.2]{Kuz:quadric} and
using that \(Z = H - h\) from Lemma
\ref{exceptional-divisor-Z}\ref{exceptional-divisor-Z.divisor-class} therefore
gives:

\begin{lemma}
\label{sod-Z}
There is a semiorthogonal decomposition
\[
  \Db(Z) = \langle
  b_Z^*\Db(L) \otimes \sO_Z(2Z),
  b_Z^*\Db(L) \otimes \sO_Z(Z),
  \mathcal{D}
  \rangle
\]
where \(\mathcal{D} = \langle b_Z^*\Db(L), \Db(L,\mathcal{B}_0') \rangle\)
for \(\mathcal{B}_0'\) the even parts of a sheaf of Clifford algebras on $L$.
\qed
\end{lemma}


Now we can give the promised resolution of $\tKu(X)$.
\begin{lemma}
\label{lemma-DbtX}
There is a semiorthogonal decomposition
\begin{equation}
\label{tKuX}
\begin{multlined}
\Db(\widetilde X) =
\langle
  i_*b_Z^*\Db(L) \otimes \sO_Z(2Z),
  i_*b_Z^*\Db(L) \otimes \sO_Z(Z), \\
  \widetilde\Ku(X), 
  \sO_{\widetilde X},
  \sO_{\widetilde X}(H),
  \sO_{\widetilde X}(2H),
  \sO_{\widetilde X}(3H)
\rangle,
\end{multlined}
\end{equation} 
where \(\tKu(X)\) is a crepant categorical resolution of singularities of
\(\Ku(X)\). More precisely, writing \(\Ku(X)^{\perf}\) for the subcategory of
\(\Ku(X)\) consisting of perfect complexes, pullback and pushforward along
$b_X \colon \tX \to X$ restrict to functors
\begin{equation*}
b_X^* \colon \Ku(X)^{\perf} \to \tKu(X)
\qquad \text{and} \qquad
b_{X*} \colon \tKu(X) \to \Ku(X)
\end{equation*}
which are mutually left and right adjoint.
\end{lemma}

\begin{proof}
Apply \cite[Theorem 1]{Kuz:cat-res} to the resolution of singularities
\(b_X \colon \widetilde X \to X\) and the decomposition of \(\Db(Z)\) from
Lemma \ref{sod-Z}. For this, Lemma \ref{exceptional-divisor-Z} implies
that \(nZ = -nH + nh\) restricts to \(\sO(n)\) along the quadric threefold
fibers of \(b_Z \colon Z \to L\), from which it follows from the
theorem on formal functions that \(b_{X*}\sO_{\widetilde{X}} \cong \sO_X\).
It is then easy to see that the assumptions of Kuznetsov's theorem are
satisfied, so that it gives a semiorthogonal decomposition
\[
\Db(\widetilde X) =
\langle
  i_*b_Z^*\Db(L) \otimes \sO_Z(2Z),
  i_*b_Z^*\Db(L) \otimes \sO_Z(Z),
  \widetilde{\mathcal{D}}
\rangle
\]
where \(\widetilde{\mathcal{D}}\) is a crepant categorical resolution of
singularities of \(\Db(X)\). In particular, $b_X^*$ fully faithfully embeds the
category of perfect complexes on $X$ into $\widetilde{\mathcal{D}}$, so
$\widetilde{\mathcal{D}}$ contains the objects
\(b_X^*\sO_X(mH) = \sO_{\widetilde X}(mH)\) for \(m = 0,1,2,3\), and they
remain a semiorthogonal exceptional collection. Therefore, we obtain a
semiorthogonal decomposition
\[
\widetilde{\mathcal{D}} =
\langle
  \widetilde\Ku(X),
  \sO_{\widetilde X},
  \sO_{\widetilde X}(H),
  \sO_{\widetilde X}(2H),
  \sO_{\widetilde X}(3H)
\rangle,
\]
where \(\widetilde{\Ku}(X)\) is the right orthogonal to the subcategory
of \(\widetilde{\mathcal{D}}\) generated by the displayed line bundles.
Putting these decompositions together now gives the statement.
\end{proof}

\subsection{Clifford algebra description of \(\tKu(X)\)}\label{S:clifford}
Since \(\pi \colon \widetilde X \to \PP\widebar{V}\) is a quadric surface
fibration by Lemma~\ref{quadric-bundle},
\cite[Theorem 4.2]{Kuz:quadric} gives a semiorthogonal decomposition
\begin{equation}
\label{DbtX2}
\Db(\tX) =
\langle
  \Db(\PP\widebar{V}, \mathcal{B}_0),
  \pi^*\Db(\PP\widebar{V}) \otimes \sO_{\widetilde X},
  \pi^*\Db(\PP\widebar{V}) \otimes \sO_{\widetilde X}(H)
\rangle,
\end{equation}
where \(\cB_0\) is the sheaf on \(\PP\widebar{V}\) of even Clifford algebras associated
with \(\pi \colon \widetilde{X} \to \PP\widebar{V}\). This subsection aims to prove:

\begin{prop}
\label{prop-Ku-Cl}
There is an equivalence of categories $\tKu(X) \simeq \Db(\PP\widebar{V}, \mathcal{B}_0)$.
\end{prop}

We shall compare the semiorthogonal decomposition from Lemma \ref{lemma-DbtX}
with that of \eqref{DbtX2} via a sequence of mutations. Recall that given a
semiorthogonal decomposition
\[ \mathcal{T} = \langle \mathcal{A}_1, \mathcal{A}_2, \ldots, \mathcal{A}_n \rangle \]
of a triangulated category with admissible components, there are functors
\(\mathbf{L}_{\mathcal{A}_i}, \mathbf{R}_{\mathcal{A}_j} \colon \mathcal{T} \to \mathcal{T}\),
for \(1 \leq i \leq n-1\) and \(2 \leq j \leq n\), which give
semiorthogonal decompositions of \(\mathcal{T}\) of the form
\[
  \langle
    \mathcal{A}_1,
    \ldots,
    \mathcal{A}_{i-1},
    \mathbf{L}_{\mathcal{A}_i}(\mathcal{A}_{i+1}),
    \mathcal{A}_i,
    \ldots,
    \mathcal{A}_n
  \rangle
  \;\;\;\text{and}\;\;\;
  \langle
    \mathcal{A}_1,
    \ldots,
    \mathcal{A}_j,
    \mathbf{R}_{\mathcal{A}_j}(\mathcal{A}_{j-1}),
    \mathcal{A}_{j+1},
    \ldots,
    \mathcal{A}_n
  \rangle,
\]
called the \emph{left mutation through \(\mathcal{A}_i\)} and
\emph{right mutation through \(\mathcal{A}_j\)}, respectively.
The following summarizes some basic properties of these functors that we will
freely use below; see~\cite{Bondal:mutation, BK:mutation, Kuz:cubic} for
details.

\begin{lemma}
\label{lemma-mutations}
Let $\mathcal{T} = \langle \mathcal{A}_1, \mathcal{A}_2, \ldots, \mathcal{A}_n \rangle$ be a semiorthogonal
decomposition with admissible components.
\begin{enumerate}
\item If \(\mathcal{A}_k\) and \(\mathcal{A}_{k+1}\)
are completely orthogonal, meaning $\Hom^\bullet(E, F) = 0$ for
\(E \in \mathcal{A}_{k}\) and \(F \in \mathcal{A}_{k+1}\), then
\[
  \mathbf{L}_{\mathcal{A}_k}(\mathcal{A}_{k+1}) = \mathcal{A}_{k+1}
    \quad\text{and}\quad
  \mathbf{R}_{\mathcal{A}_{k+1}}(\mathcal{A}_k) = \mathcal{A}_k .
\]
\item If $\cA_{k}$ is generated by an exceptional object $E$, then the
  associated mutation functors $\mathbf{L}_{E}$ and $\mathbf{R}_{E}$ are given by
\begin{equation*}
  \mathbf{L}_E(F) = \mathrm{Cone}(\Hom^\bullet(E,F) \otimes E \to F)
  ~ \text{ and } ~
  \mathbf{R}_E(F) = \mathrm{Cone}(F \to \Hom^\bullet(F,E)^\vee \otimes E)[-1].
\end{equation*}

\item If \(\mathcal{T} = \Db(Y)\) for a smooth projective variety \(Y\), then
\begin{equation*}
\mathbf{L}_{\langle \cA_1, \dots, \cA_{n-1} \rangle}(\cA_n) = \cA_n \otimes \omega_Y
  \quad\text{and}\quad
\mathbf{R}_{\langle \cA_2, \dots, \cA_{n} \rangle}(\cA_1) = \cA_1 \otimes \omega_Y^{-1}. 
\end{equation*}
\end{enumerate}
\end{lemma}

We now proceed with the proof of Proposition~\ref{prop-Ku-Cl}. Using the
standard Beilinson decomposition for the derived category of projective space,
the categories to the right of  $\Db(\PP\widebar{V}, \mathcal{B}_0)$ in~\eqref{DbtX2}
are generated by the exceptional collection
\[
  \tag*{\(\circled{0}\)}
  \langle
  \sO_{\widetilde X}(-h), \sO_{\widetilde X},     \sO_{\widetilde X}(h), \sO_{\widetilde X}(2h),
  \sO_{\widetilde X}(H),  \sO_{\widetilde X}(H + h), \sO_{\widetilde X}(H + 2h), \sO_{\widetilde X}(H + 3h)
  \rangle.
\]
Similarly, after mutating $\tKu(X)$ to the far left of the decomposition from Lemma~\ref{lemma-DbtX},
the categories to its right are generated by the exceptional collection
\[
  \tag*{\circled{8}}
  \langle
  i_*\sO_Z(2Z-2H), i_*\sO_Z(2Z -H), i_*\sO_Z(Z), i_*\sO_Z(Z + H),
  \sO_{\widetilde X}, \sO_{\widetilde X}(H), \sO_{\widetilde X}(2H), \sO_{\widetilde X}(3H)
  \rangle .
\]
It suffices to find a sequence of mutations which takes the
exceptional collection  \(\circled{0}\) to  \(\circled{8}\).
We explain the steps below; see also Figure~\ref{mutation-summary} for a summary.

\begin{figure}
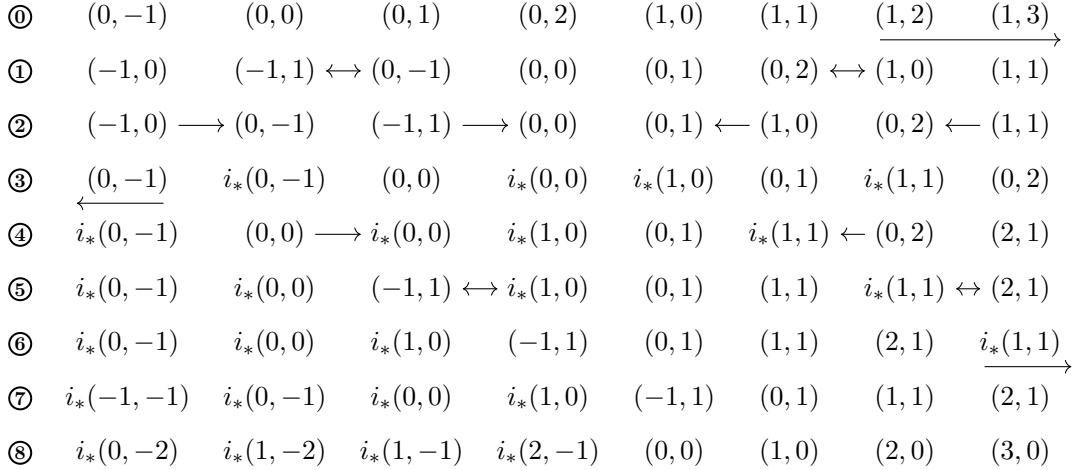

  \mutation
  \caption{\small%
  This diagram summarizes the sequence of mutations we perform to
  transform the exceptional collection coming from~\eqref{DbtX2}
to the exceptional collection coming from
  the decomposition of Lemma~\ref{lemma-DbtX}.
  The symbol \((a,b)\) represents the object \(\sO_{\widetilde X}(aH + bh)\), and
  \(i_*(a,b)\) represents the object \(i_*\sO_Z(aH + bh)\).
  An arrow indicates that the next row is obtained by mutating the object at
  the tail of the arrow through the object at the head of the arrow.
  Arrows that point off the sides indicate the application of Serre duality
  to flip underlined objects to the other side.}
  \label{mutation-summary}
\end{figure}

\smallskip
\noindent\textbf{Step 1.}
Mutate the subcategory
\(\langle \sO_{\widetilde X}(H + 2h), \sO_{\widetilde X}(H + 3h) \rangle\) on the
right end of \(\circled{0}\) to the far left.
Left mutating \(\Db(\PP\widebar{V},\mathcal{B}_0)\) through the resulting
category \(\langle \sO_{\widetilde X}(-H), \sO_{\widetilde X}(-H + h) \rangle\) results
in the decomposition with exceptional objects
\[
  \tag*{\circled{1}}
  \langle
  \sO_{\widetilde X}(-H), \sO_{\widetilde X}(-H + h),
  \sO_{\widetilde X}(-h), \sO_{\widetilde X},     \sO_{\widetilde X}(h), \sO_{\widetilde X}(2h),
  \sO_{\widetilde X}(H),  \sO_{\widetilde X}(H + h)
  \rangle.
\]

\noindent\textbf{Step 2.}
The pairs \(\langle \sO_{\widetilde X}(-H + h), \sO_{\widetilde X}(-h) \rangle\) and
\(\langle \sO_{\widetilde X}(2h), \sO_{\widetilde X}(H) \rangle\) are completely orthogonal.
Indeed, by Lemma~\ref{quadric-bundle} we have
\[
\pi_*\sO_{\tX}(H-2h) \cong
\mathcal{E}^\vee(-2h) \cong
\sO_{\PP\widebar{V}}(-h) \oplus \sO_{\PP\widebar{V}}(-2h)^{\oplus 2} \oplus \sO_{\PP\widebar{V}}(-3h) ,
\]
and hence $\rH^\bullet(\sO_{\tX}(H-2h)) = 0$.
Simultaneously transposing these pairs yield the collection
\[
  \tag*{\circled{2}}
  \langle
  \sO_{\widetilde X}(-H), \sO_{\widetilde X}(-h),
  \sO_{\widetilde X}(-H + h),
  \sO_{\widetilde X},     \sO_{\widetilde X}(h), \sO_{\widetilde X}(H),
  \sO_{\widetilde X}(2h), \sO_{\widetilde X}(H + h)
  \rangle.
\]

\noindent\textbf{Step 3.}
Simultaneously perform right mutations in the pairs
\begin{equation*}
\langle \sO_{\widetilde X}(-H), \sO_{\widetilde X}(-h)\rangle \quad \text{and} \quad
\langle \sO_{\widetilde X}(-H + h), \sO_{\widetilde X} \rangle,
\end{equation*}
and left mutations in the pairs
\begin{equation*}
\langle \sO_{\widetilde X}(h), \sO_{\widetilde X}(H) \rangle \quad \text{and} \quad
\langle \sO_{\widetilde X}(2h), \sO_{\widetilde X}(H + h)\rangle.
\end{equation*}
For each pair, the space of morphisms from the left object to the right is
$\rH^\bullet(\sO_{\widetilde X}(H - h)) = k[0]$, which can be computed as in
the previous step.
The nonzero section corresponds to the equation of $Z = H-h$.
Thus each mutation gives the structure sheaf of \(Z\), possibly with a twist;
for example,
\[
  \mathbf{L}_{\sO_{\widetilde X}(h)}(\sO_{\widetilde X}(H)) \coloneqq
  \mathrm{Cone}(\Hom^\bullet(\sO_{\widetilde X}(h), \sO_{\widetilde X}(H)) \otimes \sO_{\widetilde X}(h) \to \sO_{\widetilde X}(H)) \cong
  i_*\sO_Z(H).
\]
Similarly computing for the others finally yields the collection
\[
\tag*{\circled{3}}
\langle
\sO_{\widetilde X}(-h), i_*\sO_{Z}(- h), \sO_{\widetilde X}, i_*\sO_{Z},     i_*\sO_{Z}(H), \sO_{\widetilde X}(h),
i_*\sO_{Z}(H+h),  \sO_{\widetilde X}(2h)
\rangle.
\]

\noindent\textbf{Step 4.}
Right mutate \(\Db(\PP\widebar{V}, \mathcal{B}_0)\) through
\(\sO_{\widetilde X}(-h)\), and then mutate
\(\sO_{\widetilde X}(-h)\) to the far right side, resulting in
\[
\tag*{\circled{4}}
\langle
i_*\sO_{Z}(- h), \sO_{\widetilde X}, i_*\sO_{Z},     i_*\sO_{Z}(H), \sO_{\widetilde X}(h),
i_*\sO_{Z}(H+h),  \sO_{\widetilde X}(2h), \sO_{\widetilde X}(2H + h)
\rangle.
\]

\noindent\textbf{Step 5.}
Simultaneously right mutate \(\sO_{\widetilde X}\) through \(i_*\sO_Z\), and
left mutate \(\sO_{\widetilde X}(2h)\) through \(i_*\sO_Z(H+h)\).
For the right mutation, we have
\[
  \Hom^\bullet(\sO_{\widetilde X}, i_*\sO_Z)
  = \rH^\bullet(\sO_Z)
  = \rH^\bullet(\sO_E \oplus \sO_E(-H-h)) = k[0]
\]
since \(Z \to E \cong \PP\widebar{V} \times L\) is a double cover branched
along a divisor of class \(2H + 2h\) by Lemma
\ref{exceptional-divisor-Z}\ref{exceptional-divisor-Z.double-cover}. Thus the
only morphism is the canonical quotient map, so
\[
\mathbf{R}_{i_*\sO_Z}(\sO_{\widetilde X})
= \sO_{\widetilde X}(-Z)
= \sO_{\widetilde X}(-H + h),
\]
using Lemma~\ref{exceptional-divisor-Z}\ref{exceptional-divisor-Z.divisor-class}.
Likewise, for the left mutation, by Grothendieck duality for $i \colon Z \to \tX$,
\begin{align*}
\Hom^\bullet(i_*\sO_Z(H+h), \sO_{\widetilde X}(2h))
& = \Hom^\bullet(\sO_Z(H+h), \sO_Z(Z + 2h)[-1]) \\
& = \rH^\bullet(\sO_Z[-1]) \\
& = k[-1].
\end{align*}
This map corresponds to the exact sequence
\begin{equation*}
0 \to \sO_{\tX}(2h) \to \sO_{\tX}(H+h) \to i_*\sO_Z(H+h) \to 0,
\end{equation*}
and therefore
\[ \mathbf{L}_{i_*\sO_Z(H+h)}(\sO_{\widetilde X}(2h)) = \sO_{\widetilde X}(H + h). \]
In total, the exceptional collection has now become
\[
\tag*{\circled{5}}
\langle
i_*\sO_{Z}(-h), i_*\sO_{Z}, \sO_{\widetilde X}(-H+h), i_*\sO_{Z}(H), \sO_{\widetilde X}(h), \sO_{\widetilde X}(H + h),
i_*\sO_{Z}(H+h), \sO_{\widetilde X}(2H + h)
\rangle.
\]

\noindent\textbf{Step 6.}
As \(Z \to E\) is a double cover branched along a divisor of class \(2H+2h\),
\[
\mathrm{H}^\bullet(\sO_Z(2H-h))
= \mathrm{H}^\bullet(\sO_E(2H-h) \oplus \sO_E(H-2h))
= 0
\]
since the summands in the middle restrict to \(\sO(-1)\) and \(\sO(-2)\)
along the fibers of the \(\PP^3\)-bundle \(E \to L\). Thus a computation as in
the previous step shows that the pairs
\(\langle \sO_{\widetilde X}(-H+h), i_*\sO_Z(H)\rangle\) and
\(\langle i_*\sO_Z(H+h), \sO_{\widetilde X}(2H + h) \rangle\) are completely
orthogonal.
Transposing the objects in each pair results in the collection
\[
\tag*{\circled{6}}
\langle
i_*\sO_{Z}(-h), i_*\sO_{Z}, i_*\sO_Z(H), \sO_{\widetilde X}(-H+h), \sO_{\widetilde X}(h), \sO_{\widetilde X}(H + h),
\sO_{\widetilde X}(2H + h), i_*\sO_{Z}(H+h)
\rangle.
\]

\noindent\textbf{Step 7.}
Mutate \(i_*\sO_Z(H + h)\) to the far left, and left mutate \(\Db(\PP\widebar{V},\mathcal{B}_0)\) through
the resulting object \(i_*\sO_Z(-H-h)\).
This yields the collection
\[
\tag*{\circled{7}}
\langle
i_*\sO_{Z}(-H-h),
i_*\sO_{Z}(-h), i_*\sO_{Z}, i_*\sO_Z(H), \sO_{\widetilde X}(-H+h), \sO_{\widetilde X}(h), \sO_{\widetilde X}(H + h),
\sO_{\widetilde X}(2H + h)  \rangle.
\]

\noindent\textbf{Step 8.}
Finally, twist the collection \(\circled{7}\) by \(\sO_{\widetilde X}(H-h)\).
The resulting collection is exactly that appearing in \(\circled{8}\) above.
This completes the proof of Proposition~\ref{prop-Ku-Cl}.
\(\hfill\qed\)

\subsection{A twisted Calabi--Yau threefold}
\label{CY3}
Combining Lemma~\ref{lemma-DbtX} and Proposition~\ref{prop-Ku-Cl} shows that
$\Ku(X)$ admits a crepant resolution by the category
$\Db(\PP\widebar{V}, \mathcal{B}_0)$. Thus, to prove
Theorem~\ref{main-theorem}, it suffices to
identify $\Db(\PP\widebar{V}, \mathcal{B}_0)$ with the twisted derived category
of a non-projective Calabi--Yau threefold, where the twist is by a nontrivial Brauer class. 
Below, we show this 
using the results of~\cite{Kuz:lines}, 
modulo the non-projectivity of the threefold which we show in \S\ref{general-non-projective}. 

It will be convenient to work more generally and to consider any quadric
surface bundle \(Q \to \PP\widebar{V}\) of class \(2H+2h\) in
\(\PP\mathcal{E}\) with smooth total space and discriminant locus \(D\) an
octic surface with singular locus \(D_0\) consisting of only \(72\) nodes, as in Lemma
\ref{discr}\ref{discr.sings}; call any such quadric surface bundle \emph{good}.
Writing \(\mathcal{B}_0\) for its sheaf of even Clifford algebras on
\(\PP\widebar{V}\), we have:

\begin{lemma}
\label{general-rationality}
If \(Q \to \PP\widebar{V}\) is good, then there is an equivalence of categories
\[
\Db(\PP\widebar{V}, \mathcal{B}_0) \simeq
\Db(W^+,\mathcal{A}^+)
\]
where \(W^+\) is a \(3\)-dimensional Calabi--Yau algebraic space with an
Azumaya algebra \(\mathcal{A}^+\). Moreover, the Brauer class of
\(\mathcal{A}^+\) is trivial if and only if \(Q \to \PP\widebar{V}\) has a rational 
section.
\end{lemma}

\begin{proof}
This can be extracted directly from the results of \cite{Kuz:lines}. Choices,
however, need to be made in the construction of \(W^+\), and we will require a
specific set of choices later on, so we detail the construction here:

Let \(\mu \colon M \to \PP\widebar{V}\) be the relative Fano scheme of lines of
\(Q \to \PP\widebar{V}\). The fiber of \(\mu\) over a point of
\(\PP\widebar{V} \setminus D\) is $\PP^1 \sqcup \PP^1$, corresponding to the
two rulings on the smooth quadric surface fiber of \(Q \to \PP\widebar{V}\);
over a point of \(D \setminus D_0\), it is a smooth conic with multiplicity
\(2\); and over a point of \(D_0\), it is two copies of \(\PP^2\) glued at
a single point. Thus the Stein factorization of $\mu$ gives a double covering
\(\tau \colon W \to \PP\widebar{V}\) branched along \(D\). When \(Q \to
\PP\widebar{V}\) is good, then \(W\) has only \(72\) nodes as singularities,
exactly over \(D_0\).

Now~\cite[Section 4]{Kuz:lines} gives a diagram
\[
  \begin{tikzcd}
    M \rar[leftrightarrow, dashed,"\text{flip}"'] \dar & M^+ \dar \\
    W & \lar W^+
  \end{tikzcd}
\]
where \(M^+\) is an algebraic space obtained as a flip of \(M\), and
\(W^+ \to W\) is a small resolution of singularities. Since \(D\) is an octic,
\(W\) and \(W^+\) have trivial canonical bundles. By~\cite[Propositions 4.4 and
5.5]{Kuz:lines}, the morphism \(M^+ \to W^+\) is a \(\PP^1\)-bundle with Brauer
class represented by an explicit Azumaya algebra \(\mathcal{A}^+\) on \(W^+\).
The calculations on \cite[p.670]{Kuz:lines} then give the equivalence of
categories
\[
\Db(\PP\widebar{V},\mathcal{B}_0) \simeq
\Db(W^+,\mathcal{A}^+).
\]

The Brauer class of \(\mathcal{A}^+\) vanishes if and only if the
\(\PP^1\)-bundle \(M^+ \to W^+\) admits a rational section. Upon undoing the
birational modifications, this is equivalent to \(M \to W\) admitting a
rational section. Since a rational section of \(M \to W\) gives one line
in each of the two rulings on the generic fiber of \(Q \to \PP\widebar{V}\),
taking their intersection gives a rational section of the quadric surface
bundle \(Q \to \PP\widebar{V}\). Conversely, a rational section of
\(Q \to \PP\widebar{V}\) distinguishes the unique pair of lines
passing through this point, and thus gives a rational section of \(M \to W\).
\end{proof}

Suppose furthermore that the good quadric bundle \(Q \to \PP\widebar{V}\)
arises as a double cover \(Q \to \PP\mathcal{F}\) of a \(\PP^2\)-bundle
\(\PP\mathcal{F} \to \PP\widebar{V}\) branched over a conic bundle
\(C \subset \PP\mathcal{F}\) of class \(2H + 2h\). This is the case, for
instance, with the quadric surface bundles \(\widetilde{X} \to \PP\widebar{V}\)
arising from quartics fourfolds \(Y\) from the linear system \(\mathfrak{a}\).
In this case, we can show that the Brauer class of the Azumaya algebra
\(\mathcal{A}^+\) appearing above is nontrivial:

\begin{lemma}\label{no-rational-section}
In the situation above, \(Q \to \PP\widebar{V}\) does not admit a rational
section, and thus the Brauer class of $\cA^{+}$ is nontrivial. 
\end{lemma}

\begin{proof}
By the closing arguments of Lemma \ref{general-rationality}, it suffices
to show that \(M \to W\) does not admit a rational section. Begin by observing
that \(M\) is birational over $W$ to the pullback \(C_W \to W\) of the
conic bundle \(C \to \PP\widebar{V}\). Work over the generic points
\(\xi \in \PP\widebar{V}\) and \(\eta \in W\). Points of \(M_\eta\) over
\(\eta\) consist of a choice of ruling of the quadric surface \(Q_\xi\) and a
line \(\ell\) in that ruling. The double cover \(Q_\xi \to \PP\mathcal{F}_\xi\)
is induced by linear projection from a point outside of \(Q_\xi\), so the line
\(\ell\) is mapped to a line \(\ell'\) in the plane \(\PP\mathcal{F}_\xi\). The
ramification locus of the cover is a hyperplane section of \(Q_\xi\), and so
\(\ell'\) intersects the branch locus \(C_\xi\) at a single point: In other
words, \(\ell'\) must be a tangent line to \(C_\xi\). This gives a
morphism from \(M_\eta\) to the projective dual of \(C_\xi\), which may be
identified via the Gauss map with \(C_\xi\) itself. Since only one line from
each ruling passes through a given point of a quadric surface, this implies
that the induced morphism \(M_\eta \to C_\eta\) upon base change to \(\eta\) is
an isomorphism.

It remains to show that the conic bundle \(C_W \to W\) does not
admit a rational section. Since \(W \to \PP\widebar{V}\) is a finite morphism,
\(C_W\) is an ample divisor in the \(\PP^2\)-bundle \(\PP\mathcal{F}_W\). Up to
verifying genericity hypotheses for \(C_W\), the version of the
Grothendieck--Lefschetz theorem and its proof in \cite[Theorem 1]{RS:GL} shows
that the restriction map
\[
\operatorname{Cl}(\PP\mathcal{F}_W) \to
\operatorname{Cl}(C_W)
\]
on divisor class groups is an isomorphism,
at which point the result follows since the degree over \(W\) of all
divisors on \(C_W\) is even, whence there can be no rational section. The
genericity conditions required on \(C_W\) may be extracted from analyzing the
proof and can be found in \cite[pp.564, 572, and 577]{RS:GL}. They are as
follows:
\begin{enumerate}
\item\label{no-rational-section.singular-locus}
\(\operatorname{Sing}(C_W) = \operatorname{Sing}(\PP\mathcal{F}_W) \cap C_W\); and
\item\label{no-rational-section.exceptionals}
the exceptional divisors in the resolution of \(\PP\mathcal{F}_W\) which are
contracted to points and curves are disjoint from the strict transform of
\(C_W\).
\end{enumerate}
Point \ref{no-rational-section.singular-locus} is clear, since the singular
loci are but the fibers over the nodes of \(W\). Point
\ref{no-rational-section.exceptionals} is because the singularities of
\(\PP\mathcal{F}_W\) are products of a node and \(\PP^2\), are resolved upon a
single blow up, and so every exceptional divisor contracts to a center of
dimension at least \(2\).
\end{proof}

\begin{remark}
\label{remark-nonvanishing}
Although Lemma \ref{no-rational-section} applies only to a special family of
good quadric bundles, it seems reasonable to expect that the same should hold
true for all such quadrics: namely, we expect that no good quadric bundle can
admit a rational section.
Indeed, the exponential sequence shows that the Brauer group of a strict
(meaning that $h^{i}(\sO) = 0$ for $0 < i < n$) complex Calabi--Yau \(n\)-fold
is nothing but the torsion in degree \(3\) cohomology when \(n \geq 3\), and
hence a topological invariant. Thus, if it were possible to construct the pair
$(W^{+}, \cA^{+})$ in families, then it would follow that the Brauer class of
$\cA^{+}$ either vanishes everywhere or nowhere, and by
Lemma~\ref{general-rationality} the vanishing is equivalent to the existence of
rational section of $Q \to \PP\widebar{V}$. This argument is not complete,
because a priori it is not clear that the choice of the small resolution of
nodal singularities in the construction of $(W^+, \cA^{+})$ can be made
simultaneously in a family.
\end{remark}

\subsection{Non-projectivity}\label{general-non-projective}
To complete the proof of Theorem~\ref{main-theorem}, we 
just need to show that the Calabi--Yau algebraic space \(W^+\)
appearing in \S\ref{CY3} is not projective. 
We show that it has no ample
divisors via the following non-projectivity criterion:

\begin{lemma}\label{non-projectivity-criterion}
Let \(T\) be a projective threefold with only nodal singularities.
If \(\delta \coloneqq b_4(T) - b_2(T)\) vanishes, then no small
resolution of \(T\) can be projective.
\end{lemma}

\begin{proof}
Let \(T^+ \to T\) be any small resolution of \(T\). Cohomology of the pair
\(T^+\) with the union of its exceptional curves \(C_1,\ldots,C_n\) gives a
sequence
\[
0 \to
\mathrm{H}^2(T) \to
\mathrm{H}^2(T^+) \to
\bigoplus\nolimits_{i = 1}^n \mathrm{H}^2(C_i) \to
\mathrm{H}^3(T) \to
\mathrm{H}^3(T^+) \to
0
\]
and an isomorphism \(\mathrm{H}^4(T) \cong \mathrm{H}^4(T^+)\). Thus
\(\delta = 0\) if and only if any class in \(\mathrm{H}^2(T^+)\) restricts
trivially to \(\mathrm{H}^2(C_i)\) for all \(i = 1,\ldots,n\). In particular,
each of the curves \(C_i\) are numerically trivial in \(T^+\), and so there
are no ample divisors. See \cite[\S1]{Clemens}, \cite[pp.5--6 and Chapter
III]{Werner}, and \cite[p.43]{Addington:Thesis} for more.
\end{proof}

The integer \(\delta\) is called the \emph{defect} of the nodal threefold.
In the case \(T\) is a double solid, Clemens provides a formula in
\cite[(3.17)]{Clemens} for \(\delta\) in terms of the number of independent
conditions on certain polynomials imposed by the position of the nodes; see
also \cite{Cynk:Defect} for a generalization and an algebraic proof.

\begin{lemma}\label{clemens-defect-formula}
If \(T\) is a double cover of \(\PP^3\) branched along a nodal surface \(B\)
of degree \(2d\), then, writing \(\mathcal{I}\) for the ideal sheaf of the
nodes of \(B\) viewed as a subscheme of \(\PP^3\),
\[
\delta = \dim\mathrm{H}^1(\PP^3,\mathcal{I} \otimes \sO_{\PP^3}(3d-4)).
\]
\end{lemma}

In the situation at hand, the nodes of the discriminant surface parameterize
points where a bilinear form has corank at least \(2\), so its ideal sheaf in
\(\PP\widebar{V}\) is locally generated by the minors of a symmetric matrix.
Such an ideal can be resolved as follows:

\begin{lemma}\label{non-projectivity-resolution}
Let \(\varphi \colon \mathcal{V}^\vee \to \mathcal{V} \otimes \mathcal{L}\) be
a symmetric morphism between locally free modules of rank \(r\) on a locally
Noetherian scheme, where $\mathcal{L}$ is a line bundle. Then there is a complex
\[
0 \to
\mathcal{L}^{\vee, \otimes 2} \otimes \wedge^2 \mathcal{V}^\vee \to
\mathcal{L}^\vee \otimes (\mathcal{V}^\vee \otimes \mathcal{V})_0 \to
\Sym^2(\mathcal{V}) \to
\mathcal{I} \otimes \mathcal{L}^{\otimes r-1} \otimes \det(\mathcal{V})^{\otimes 2} \to
0
\]
where
\((\mathcal{V}^\vee \otimes \mathcal{V})_0 \coloneqq \ker(\operatorname{ev} \colon \mathcal{V}^\vee \otimes \mathcal{V} \to \sO_S)\)
and \(\mathcal{I}\) is the sheaf of ideals locally generated by the size
\(r-1\) minors of \(\varphi\). If \(\mathcal{I}\) furthermore has its maximal
depth \(3\), then the complex is exact.
\end{lemma}

\begin{proof}
Up to twisting by a power of \(\mathcal{L}\), the first three terms
essentially comprise of a generalized Eagon--Northcott complex associated with
\(\varphi \colon \mathcal{V}^\vee \to \mathcal{V} \otimes \mathcal{L}\),
see \cite[\((\mathrm{EN}_2)\) on p.323]{Lazarsfeld:PositivityI}. The point
here is to identify the cokernel on the right, and this is done locally by
\cite[Theorem 3.1]{Joz} and \cite{GT}. Therefore it remains to
globalize J{\'o}zefiak's description of the rightmost differential: View the
map \(\wedge^{r-1} \varphi\) as a bilinear form
\[
\wedge^{r-1} \mathcal{V}^\vee \otimes
\wedge^{r-1} \mathcal{V}^\vee \to
\mathcal{L}^{\otimes r-1}.
\]
Locally, this is given by a matrix consisting of size \(r-1\) minors of
\(\varphi\), so this map is symmetric, and has image the twisted ideal sheaf
\(\mathcal{I} \otimes \mathcal{L}^{\otimes r-1}\). Upon identifying
\(\wedge^{r-1} \mathcal{V}^\vee\) with \(\mathcal{V} \otimes \det(\mathcal{V}^\vee)\)
via the isomorphism induced by wedge products, this gives a surjective map
\[
\Sym^2(\mathcal{V}) \otimes \det(\mathcal{V}^\vee)^{\otimes 2} \to
\mathcal{I} \otimes \mathcal{L}^{\otimes r-1}.
\]
Twisting by \(\det(\mathcal{V})^{\otimes 2}\) gives the desired map.
\end{proof}

These tools now give a straightforward way to show that the algebraic space
\(W^+\) associated with a good quadric bundle \(Q \to \PP\widebar{V}\) cannot
be projective:

\begin{prop}\label{non-projectivity-result}
No small resolution of the double cover \(\tau \colon W \to \PP\widebar{V}\)
is projective.
\end{prop}

\begin{proof}
The ideal sheaf \(\mathcal{I}\) of the nodes of \(D\), viewed as a subscheme of
\(\PP\widebar{V}\), is locally generated by the size \(3\) minors of the
symmetric morphism \(\mathcal{E} \to \mathcal{E}^\vee \otimes
\sO_{\PP\widebar{V}}(2h)\) associated with the bilinear form defining \(Q\).
Since \(\mathcal{I}\) defines a set of reduced points in \(\PP\widebar{V}\), it
has depth \(3\), so Lemma \ref{non-projectivity-resolution} provides a
resolution of \(\mathcal{I} \otimes \sO_{\PP\widebar{V}}(8h)\) of the form
\begin{multline*}
0
\to
\sO(-3)^2 \oplus
\sO(-2)^2 \oplus
\sO(-1)^2
\to
\sO(-2)   \oplus
\sO(-1)^4 \oplus
\sO^5     \oplus
\sO(1)^4  \oplus
\sO(2) \\
\to
\sO       \oplus
\sO(1)^2  \oplus
\sO(2)^4  \oplus
\sO(3)^2  \oplus
\sO(4)
\to
\mathcal{I} \otimes \sO(8)
\to
0,
\end{multline*}
where we abbreviate \(\sO_{\PP\widebar{V}}(kh)^{\oplus r}\) to \(\sO(k)^r\) for
\(k,r \in \mathbf{Z}\). Since none of the terms in the resolution have any higher
cohomology,
\(\mathrm{H}^1(\PP\widebar{V}, \mathcal{I} \otimes \sO_{\PP\widebar{V}}(8h)) = 0\),
so Lemma \ref{clemens-defect-formula} shows that the defect of \(W\) vanishes.
Lemma \ref{non-projectivity-criterion} then applies to give the result.
\end{proof}


\section{Special rational examples}
\label{section-special}
Specialize further and consider quartic double fivefolds arising from quartics
that are singular along the line \(L\), and which are tangent to the
complementary \(3\)-plane \(P\) along the smooth quadric surface \(S\). Such
quartic double fivefolds are rational since, as observed in \S\ref{quartics-quadric-bundle-sections}, \(P\) gives rise to a section of
the associated quadric surface bundle. In this section, we construct a further
crepant resolution of the Kuznetsov component of such fivefolds, and show that,
this time, it is equivalent to a geometric Calabi--Yau \(3\)-category. 
Note that Lemma~\ref{sod-X} holds for a double cover $X \to \PP^5$ branched along any quartic hypersurface, 
regardless of the singularities, so $\Ku(X)$ is indeed still defined in this more singular setting. 

Throughout, fix a general member \(Y\) of the linear system \(\mathfrak{b}\),
and continue with the notation of \S\ref{S:quartics}.

\subsection{Projection from a section}
\label{special-projection}
As in \S\ref{quartics-quadric-bundle-sections}, the quadric surface
bundle \(\pi \colon \widetilde{X} \to \PP\widebar{V}\) admits two distinguished
sections. Fix one of them, call it
\(\sigma \colon \PP\widebar{V} \to \widetilde{X}\), and let
\(\mathcal{N} \subset \mathcal{E}\) be the corresponding line subbundle.
Writing \(\widebar{\mathcal{E}} \coloneqq \mathcal{E}/\mathcal{N}\), relative
linear projection centered along \(\PP\mathcal{N}\) defines rational maps
\(\widetilde{X} \dashrightarrow \PP\widebar{\mathcal{E}}\) and
\(\PP\mathcal{E} \dashrightarrow \PP\widebar{\mathcal{E}}\),
the former birational, which are resolved on
the blow ups \(\widehat{X}\) and \(\widehat{\PP}\mathcal{E}\) along
\(\PP\mathcal{N}\). These maps fit into a commutative diagram
\[
\begin{tikzcd}
Z' \rar["i'"'] \dar
& \widehat{X} \dar["\hat{b}_{\widetilde{X}}"] \rar[hook]
& \widehat{\PP}\mathcal{E} \dar["\hat{b}\phantom{_{\widetilde{X}}}"] \rar["\hat{a}"']
& \PP\widebar{\mathcal{E}} \dar["\bar\pi"] \\
\PP\mathcal{N} \rar[hook]
& \widetilde{X} \rar[hook]
& \PP\mathcal{E} \rar
& \PP\widebar{V}
\end{tikzcd}
\]
where
\(i' \colon Z' \hookrightarrow \widehat{X}\) is the
exceptional divisor of \(\hat{b}_{\widetilde{X}}\). 
Let \(\hat{a}_{\widehat{X}} \colon \widehat{X} \to \PP\widebar{\mathcal{E}}\) denote the birational map given by the restriction of $\hat{a}$ to $\widehat{X}$. 

To describe the basic geometry of the situation, abuse notation and
write \(H\) and \(h\) for the hyperplane classes from \(\PP V\) and
\(\PP\widebar{V}\), respectively, pulled back to \(\widehat{\PP}\mathcal{E}\).
Let \(\xi\) be the relative hyperplane class
of \(\bar{\pi} \colon \PP\widebar{\mathcal{E}} \to \PP\widebar{V}\), and write
\(\hat{\pi} \coloneqq \bar{\pi} \circ \hat{a} \colon \widehat{X} \to \PP\widebar{V}\).

\begin{lemma}
\label{special-blowup-facts}
Let \(\mathcal{G} \coloneqq (\hat{a}_*\sO_{\widehat{\PP}\mathcal{E}}(H))^\vee\).
Then
\begin{enumerate}
\item\label{special-blowup-facts.locally-free}
\(\mathcal{G}\) fits into a short exact sequence
\(
0 \to
\bar{\pi}^*\mathcal{N} \to
\mathcal{G} \to
\sO_{\PP\widebar{\mathcal{E}}}(-\xi) \to
0
\);
\item\label{special-blowup-facts.proj-bundle}
\(\widehat{\PP}\mathcal{E} \to \PP\widebar{\mathcal{E}}\) is isomorphic
to the projective bundle \(\PP\mathcal{G} \to \PP\widebar{\mathcal{E}}\);
\item\label{special-blowup-facts.class}
the class of \(\widehat{X}\) in \(\PP\mathcal{G}\) is \(H + \xi + 2h\); and
\item\label{special-blowup-facts.canonical}
\(\xi = H - Z'\) and \(K_{\widehat{X}} = -H - \xi - 2h\) in \(\Pic(\widehat{X})\).
\end{enumerate}
\end{lemma}

\begin{proof}
Items \ref{special-blowup-facts.locally-free} and
\ref{special-blowup-facts.proj-bundle} are relative versions of the
corresponding parts from Lemma \ref{bl-facts}, and are standard. Let \(E'\)
be the exceptional divisor of the blow up \(\PP\mathcal{G} \to \PP\mathcal{E}\).
Then \(\xi = H - E'\) as divisor classes on \(\PP\mathcal{G}\). Since the
\(3\)-plane \(\PP\mathcal{N}\) lying at the center of the blow up
generically has multiplicity \(1\) in \(\widetilde{X}\), the class
of \(\widehat{X}\) in \(\PP\mathcal{G}\) is
\(2H + 2h - E' = H + \xi + 2h\) by Lemma
\ref{quadric-bundle}\ref{quadric-bundle.class}, yielding
\ref{special-blowup-facts.class}. From this,
\ref{special-blowup-facts.canonical} is a straightforward computation.
\end{proof}

The projection formula along \(\hat{a}\) gives a canonical isomorphism
\[
\mathrm{H}^0(\PP\mathcal{G},\sO_{\PP\mathcal{G}}(H+\xi+2h))
\cong
\mathrm{H}^0(\PP\widebar{\mathcal{E}}, \mathcal{G}^\vee(\xi+2h)).
\]
Thus an equation of \(\widehat{X}\) in \(\PP\mathcal{G}\) induces a canonical
section
\(w \colon \sO_{\PP\widebar{\mathcal{E}}} \to \mathcal{G}^\vee(\xi+2h)\)
which vanishes on the subscheme \(W^\doubleplus \subset \PP\widebar{\mathcal{E}}\) over which
\(\hat{a}_{\widehat{X}} \colon \widehat{X} \to \PP\widebar{\mathcal{E}}\) is not an
isomorphism. The scheme \(W^\doubleplus\) in fact admits a useful modular
interpretation:

\begin{lemma}
\label{special-reduction-modular}
The scheme \(W^\doubleplus\) canonically embeds into the relative Fano
scheme of lines of \(\pi \colon \widetilde{X} \to \PP\widebar{V}\)
as the subscheme of those lines incident with the section \(\sigma\).
\end{lemma}

\begin{proof}
The modular description of projective bundles identifies
\(\PP\widebar{\mathcal{E}}\) as the subscheme in the relative Fano scheme of lines of
\(\PP\mathcal{E} \to \PP\widebar{V}\) parameterizing
lines incident with the section \(\sigma\), and such that
\(\hat{a} \colon \PP\mathcal{G} \to \PP\widebar{\mathcal{E}}\) is the
universal family. Since \(W^\doubleplus\) may be characterized as the
subscheme of \(\PP\widebar{\mathcal{E}}\) over which the entire fiber of
\(\hat{a}\) is contained in \(\widehat{X}\), the result follows.
\end{proof}

This description allows us to determine the dimension of \(W^\doubleplus\), and
thereby identify the birational morphism
\(\hat{a}_{\widehat{X}} \colon \widehat{X} \to \PP\widebar{\mathcal{E}}\) with
what it naturally should be:

\begin{lemma}
\label{special-identify-as-blowup}
The morphism
\(\hat{a}_{\widehat{X}} \colon \widehat{X} \to \PP\widebar{\mathcal{E}}\) is
isomorphic to the blow up along \(W^\doubleplus\).
\end{lemma}

\begin{proof}
Observe that \(W^\doubleplus\) is of its expected dimension \(3\): Lemma
\ref{special-reduction-modular} together with Lemma
\ref{quartics-sections-mostly-smooth} implies that
\(W^\doubleplus \to \PP\widebar{V}\) is finite of degree \(2\) away from the
singularities of the discriminant surface \(D\), and otherwise has
\(1\)-dimensional fibers. The ideal sheaf \(\mathcal{I}\) of
\(W^\doubleplus\) in \(\PP\widebar{\mathcal{E}}\) therefore admits a Koszul
resolution
\[
0 \to
\sO_{\PP\widebar{\mathcal{E}}} \xrightarrow{w}
\mathcal{G}^\vee(\xi+2h) \to
\mathcal{I} \otimes \det(\mathcal{G}^\vee(\xi+2h)) \to
0.
\]
The blow up of \(\PP\widebar{V}\) along \(W^\doubleplus\) is canonically
isomorphic to the \(\mathrm{Proj}\) of the Rees algebra associated with the
twisted ideal sheaf on the right. This sequence then embeds the blow up into
\(\PP\mathcal{G}\) as the relative hyperplane corresponding to the section
\(w\). But this is precisely \(\widetilde{X}\) by the construction of the
section \(w\), yielding the result.
\end{proof}

\begin{remark}
The construction of \(W^\doubleplus \to \PP\widebar{V}\) from the quadric
surface bundle \(\pi \colon \widetilde{X} \to \PP\widebar{V}\) and the section
\(\sigma\) might be viewed as a singular variant of \emph{hyperbolic
reduction}: Typically, this is a construction that takes a flat quadric bundle
equipped with a smooth section, and produces a quadric bundle of dimension two
less whose homological properties are closely related to those of the original
quadric bundle. See, for example, \cite[\S1.4]{ABB}, \cite[\S2.3]{KS:Quadrics},
\cite{Kuznetsov:Reduction}, and \cite[\S4]{Xie:Quadrics}.
\end{remark}

\subsection{Conifold transition}
\label{special-conifold-transition}
Let \(\tau \colon W \to \PP\widebar{V}\) be, as in \S\ref{CY3}, the double
cover branched along the discriminant surface \(D\). When \(Y\) is a general
member of \(\mathfrak{b}\), \(W\) has \(90\) nodes by Lemma
\ref{discr}\ref{discr.sings-P}: \(72\) corresponding to the corank \(2\) fibers
of \(\pi\), and an additional \(18\) corresponding to those of
\(\widetilde{X}\). Construct a small resolution of \(W\) in two steps. Begin
with a small resolution \(W^+ \to W\) of the \(72\) nodes over the subscheme
\(D_0 \subset \PP\widebar{V}\) corresponding to corank \(2\) fibers of \(\pi\);
specifically, choose \(W^+\) so that the \(M^+\) in Lemma
\ref{general-rationality} is obtained by flipping the planes in \(M\)
parameterizing lines in the planes of
\(\pi^{-1}(D_0) \hookrightarrow \widetilde{X}\) \emph{not} incident with the
section \(\sigma\). Combined with Lemma \ref{quartics-sections-mostly-smooth},
this ensures that \(W^\doubleplus\), embedded in \(M\) via Lemma
\ref{special-reduction-modular}, is disjoint from indeterminancy locus of the
birational map \(M \dashrightarrow M^+\), providing an embedding
\(W^\doubleplus \hookrightarrow M^+\). Next, composing with \(M^+ \to W^+\)
provides a morphism \(W^\doubleplus \to W^+\). This resolves the remaining
nodes:

\begin{lemma}
\label{special-small-resolution}
The morphism \(W^\doubleplus \to W^+\) is a small resolution of singularities.
\end{lemma}

\begin{proof}
Lemma \ref{special-reduction-modular} implies that \(W^\doubleplus \to W^+\)
is an isomorphism away from the nodes of \(D\). Over a node \(t \in D_0\),
the arguments of \cite[\S4]{Kuz:lines} identify the map \(M^+ \to W^+\) over
\(t\) with the projection of the Hirzebruch surface \(F_1 \to \PP^1\).
Since the section \(\sigma\) passes through a smooth point of \(\pi^{-1}(t)\)
by Lemma \ref{quartics-sections-mostly-smooth}, the embedding
\(W^\doubleplus \hookrightarrow M^+\) is identified over \(t\) as the
embedding of a non-exceptional section, and so \(W^\doubleplus \to W^+\) is
an isomorphism over \(t\).

Consider now a node \(t \in D\) lying under a node \(x \in \widetilde{X}\).
Lemmas \ref{special-reduction-modular} and
\ref{quartics-sections-mostly-smooth} together with Corollary
\ref{quartics-nodes-are-cone-points} imply that the fiber of
\(W^\doubleplus \to W^+\) above \(t\) is the smooth conic at the base of the
quadric cone \(\pi^{-1}(t)\), so it remains to see that the total space of
\(W^\doubleplus\) is smooth above \(t\). The proof of Lemma
\ref{discr}\ref{discr.sings-P} implies that there exists \'etale coordinates
\((x_0,x_1,x_2)\) centered at \(t \in \PP\widebar{V}\)
such that \(\widetilde{X} \to \PP\widebar{V}\) along with its section \(\sigma\)
is, locally around \(t\), pulled back from the quadric surface bundle in
\(\mathbf{A}^3 \times \PP^3\) defined by
\[
-y_0^2 + \beta_{11} y_1^2 + \beta_{12} y_1y_2 + \beta_{22} y_2^2 + (x_1 y_1 + x_2 y_2 + x_0^2 y_3)y_3 = 0
\]
with section \((x_0: 0: 0: 1)\); here,
\(\beta_{11}, \beta_{12}, \beta_{22} \in \Gamma(\mathbf{A}^3, \sO_{\mathbf{A}^3})\)
are such that the binary quadratic form
\(\beta_{11} y_1^2 + \beta_{12} y_1y_2 + \beta_{22} y_2^2\) has rank \(2\) in a
neighbourhood of \(0 \in \mathbf{A}^3\).

Make the change in projective coordinates \(y_0 \mapsto y_0 + x_0y_3\) to
simplify the section to \((0:0:0:1)\). The equation of the quadric surface
bundle then becomes
\[
-y_0^2 + \beta_{11}y_1^2 + \beta_{12}y_1y_2 + \beta_{22}y_2^2 + (x_0y_0 + x_1y_1 + x_2y_2)y_3 = 0.
\]
Projection away from the section \((0:0:0:1)\) has the effect of
eliminating the coordinate \(y_3\), and a standard computation shows that
\(W^\doubleplus\) is, locally around \(t\), pulled back from the complete
intersection in \(\mathbf{A}^3 \times \PP^2\) given by
\[
-y_0^2 + \beta_{11} y_1^2 + \beta_{12} y_1y_2 + \beta_{22} y_2^2 =
x_0y_0 + x_1y_1 + x_2y_2 = 0.
\]
A direct computation with the Jacobian criterion now implies that the points
above the origin of \(\mathbf{A}^3\) are smooth. This implies that
\(W^\doubleplus\) is smooth above \(t\), completing the proof.
\end{proof}

Combined with Lemma \ref{special-identify-as-blowup}, this implies that
\(\widehat{X}\) is smooth, and so:

\begin{cor}
\label{special-resolution-of-tilde-X}
The morphism \(\hat{b}_{\widetilde{X}} \colon \widehat{X} \to \widetilde{X}\)
is a resolution of singularities.
\end{cor}

To finish the present discussion, consider the exceptional divisor
\(i' \colon Z' \hookrightarrow \widehat{X}\) of the blow up
\(\hat{b}_{\widetilde{X}}\). Write
\(\hat{\pi}_{Z'} \coloneqq \hat{\pi} \circ i' \colon Z' \to \PP\widebar{V}\).
Observe first that \(Z'\) itself is smooth: Away from nodes of \(\widetilde{X}\),
it is the exceptional divisor of a blow up with a smooth center. Over a node
\(x \in \widetilde{X}\), the arguments of Lemma \ref{special-small-resolution}
show that \(Z'\) is, \'etale locally around
\(t \coloneqq \hat{\pi}_{Z'}(x) \in \PP\widebar{V}\), the exceptional
divisor to the blow up of
\[
-y_0^2 + \beta_{11}y_1^2 + \beta_{12}y_1y_2 + \beta_{22}y_2^2 + x_0y_0 + x_1y_1 + x_2y_2 = 0.
\]
in \(\mathbf{A}^3 \times \mathbf{A}^3\) along \(y_0 = y_1 = y_2 = 0\). A direct
computation then shows that it is smooth. Next, the following shows that the
structure sheaf \(\sO_{Z'}\) is a relatively exceptional object over
\(\PP\widebar{V}\), and that the conormal sheaf
\(\sO_{Z'}(Z')\) has vanishing cohomology over \(\PP\widebar{V}\):

\begin{lemma}
\label{special-resolution-exceptional-divisor}
\(\hat{\pi}_{Z',*}\sO_{Z'} = \sO_{\PP\widebar{V}}[0]\) and
\(\hat{\pi}_{Z',*}\sO_{Z'}(Z') = 0\) in \(\Db(\PP\widebar{V})\).
\end{lemma}

\begin{proof}
Write the ideal sheaf sequence of \(Z'\) in \(\widehat{X}\) using Lemma
\ref{special-blowup-facts}\ref{special-blowup-facts.canonical} as
\[
0 \to
\sO_{\widehat{X}}(-H+\xi) \to
\sO_{\widehat{X}} \to
i'_*\sO_{Z'} \to
0.
\]
A straightforward computation using the ideal sheaf sequence of \(\widehat{X}\)
in \(\PP\mathcal{G}\) together with the facts of Lemma
\ref{special-blowup-facts} shows that 
\((\bar\pi \circ \hat{a}_{\widehat{X}})_* \sO_{\widehat{X}}(-H+\xi) = 0\), 
and so
\[
\hat{\pi}_{Z',*}\sO_{Z'} \cong
(\bar{\pi} \circ \hat{a}_{\widehat{X}} \circ i')_*\sO_{Z'} \cong
(\bar{\pi} \circ \hat{a}_{\widehat{X}})_*\sO_{\widehat{X}} \cong
\hat{\pi}_*\sO_{\widehat{X}} \cong
\sO_{\PP\widebar{V}}[0].
\]

For the second claim, consider the sequence 
\begin{equation*}
0 \to
\sO_{\widehat{X}} \to
\sO_{\widehat{X}}(H-\xi) \to
i'_*\sO_{Z'}(Z') \to
0
\end{equation*} 
obtained by twisting by \(Z' = H - \xi\) the ideal sheaf sequence of $Z'$ in $\widehat{X}$. 
Again, using the ideal sheaf sequence of \(\widehat{X}\)
in \(\PP\mathcal{G}\) together with the facts of Lemma
\ref{special-blowup-facts}, a straightforward computation 
shows that \((\bar\pi \circ \hat{a}_{\widehat{X}})_* \sO_{\widehat{X}}(H-\xi) \cong \sO_{\PP\widebar{V}}[0]\). 
Since also \((\bar\pi \circ \hat{a}_{\widehat{X}})_*\sO_{\widehat{X}} \cong
\sO_{\PP\widebar{V}}[0]\), we find 
\begin{equation*}
\hat{\pi}_{Z',*}\sO_{Z'}(Z') = (\bar{\pi} \circ \hat{a}_{\widehat{X}} \circ i')_*\sO_{Z'}(Z')  = 0. \qedhere
\end{equation*} 
\end{proof}

\subsection{Geometric Kuznetsov component}
\label{special-geometric-equivalence}
The geometric situation considered so far in this section is summarized by the
commutative diagram
\[
\begin{tikzcd}
Z' \ar[rr,hook, "i'"'] \ar[d]
&
& \widehat{X}
  \ar[dl, "\hat{b}_{\widetilde{X}}"]
  \ar[dd, "\hat{\pi}"]
  \ar[dr, "\phantom{\mathclap{\hat{b}_{\widehat{X}}}}\hat{a}_{\widehat{X}}"']
&
& \Lambda \ar[ll,hook'] \ar[d] \\
\PP\mathcal{N} \ar[r, hook]
& \widetilde{X} \ar[dr, "\pi"]
&
& \PP\widebar{\mathcal{E}} \ar[dl, "\bar{\pi}"']
& W^\doubleplus \ar[l,hook', "j"'] \ar[d]\\
&& \PP\widebar{V} \ar[llu,"\sigma"',"\cong",bend left=20]
& W \ar[l,"\tau"']
& W^+ \ar[l]
\end{tikzcd}
\]
where \(\Lambda \coloneqq \PP\mathcal{G}\rvert_{W^\doubleplus}\) is the exceptional
divisor of the blow up \(\hat{a}_{\widehat{X}}\). Notably, \(\widehat{X}\) is a
blow up of schemes over \(\PP\widebar{V}\) in two different ways. This
distinguishes two
\(\Db(\PP\widebar{V})\)-linear semiorthogonal decompositions of
\(\Db(\widehat{X})\), where linearity means that each semiorthogonal component
is stable under tensor products with objects in the image of \(\hat{\pi}^*
\colon \Db(\PP\widebar{V}) \to \Db(\widehat{X})\).
They are:

\begin{lemma}
\label{special-decompositions}
There are \(\Db(\PP\widebar{V})\)-linear semiorthogonal decompositions of
\(\Db(\widehat{X})\) given by
\begin{gather*}
\big\langle
\hat{\pi}^*\Db(\PP\widebar{V}) \otimes i'_*\sO_{Z'}(Z'),
\widehat{\Ku}(X),
\hat{\pi}^*\Db(\PP\widebar{V}) \otimes \sO_{\widehat{X}},
\hat{\pi}^*\Db(\PP\widebar{V}) \otimes \sO_{\widehat{X}}(H)
\rangle \;\text{and} \\
\big\langle
\hat{a}_{\widehat{X}}^*j_*\Db(W^\doubleplus) \otimes \sO_{\widehat{X}}(-\Lambda),
\hat{\pi}^*\Db(\PP\widebar{V}) \otimes \sO_{\widehat{X}},
\hat{\pi}^*\Db(\PP\widebar{V}) \otimes \sO_{\widehat{X}}(\xi),
\hat{\pi}^*\Db(\PP\widebar{V}) \otimes \sO_{\widehat{X}}(2\xi)
\big\rangle,
\end{gather*}
where \(\widehat{\Ku}(X)\) is a crepant categorical resolution of singularities
of \(\Ku(X)\) and \(\tKu(X)\).
\end{lemma}

\begin{proof}
The second semiorthogonal decomposition arises from the blow up and projective
bundle formulas upon identifying
\(\hat{a}_{\widehat{X}} \colon \widehat{X} \to \PP\widebar{\mathcal{E}}\)
as the blow up along \(W^\doubleplus\) via Lemma
\ref{special-identify-as-blowup}.

For the first decomposition, we begin by noting that some of the arguments
from \S\ref{section-general} go through in the more singular setting we
are considering now.
First, examining the arguments of \cite[\S4]{Kuz:cat-res} shows that we may
still define a subcategory $\tKu(X) \subset \Db(\widetilde{X})$ by the
semiorthogonal decomposition~\eqref{tKuX} of Lemma~\ref{lemma-DbtX}.
It will no longer be a crepant categorical resolution, since $\widetilde{X}$ is singular,
but pullback and pushforward along
$b_X \colon \tX \to X$ still restrict to functors
\begin{equation*}
b_X^* \colon \Ku(X)^{\perf} \to \tKu(X)
\qquad \text{and} \qquad
b_{X*} \colon \tKu(X) \to \Ku(X)
\end{equation*}
which are mutually left and right adjoint.
Second, Proposition~\ref{prop-Ku-Cl} holds with verbatim proof, so that there is an equivalence $\tKu(X) \simeq \Db(\PP\widebar{V}, \mathcal{B}_0)$.

Now apply
\cite[Theorem 1]{Kuz:cat-res} with the resolution of singularities
\(\hat{b}_{\widehat{X}} \colon \widehat{X} \to \widetilde{X}\) from
Corollary \ref{special-resolution-of-tilde-X}; the assumptions are still
satisfied in this situation because all the additional singularities are nodes,
whence rational. A suitable semiorthogonal decomposition of \(\Db(Z')\) is
provided by Lemma \ref{special-resolution-exceptional-divisor}, which implies
that there is a \(\Db(\PP\widebar{V})\)-linear decomposition
\[
\Db(Z') =
\langle
  \hat{\pi}_{Z'}^*\Db(\PP\widebar{V}) \otimes \sO_{Z'}(Z'),
  \mathcal{D}'
\rangle
\]
where \(\mathcal{D}'\) is the left orthogonal to
\(\hat{\pi}_{Z'}^*\Db(\PP\widebar{V}) \otimes \sO_{Z'}(Z')\),
and contains \(\hat{\pi}_{Z'}^*\Db(\PP\widebar{V}) \otimes \sO_{Z'}\).
Arguing as in Lemma \ref{lemma-DbtX} and using the decomposition~\eqref{DbtX2} now gives the second promised semiorthogonal decomposition of $\Db(\widehat{X})$ and the fact that \(\widehat{\Ku}(X)\) is a crepant categorical resolution of $\Db(\PP\widebar{V}, \mathcal{B}_0)$. 
By the statements from the previous paragraph, this implies that 
\(\widehat{\Ku}(X)\) is also a crepant categorical resolution of $\Ku(X)$ and  \(\tKu(X)\). 
\end{proof}

To identify the
crepant categorical resolution \(\widehat{\Ku}(X)\) with the geometric
Calabi--Yau category \(\Db(W^\doubleplus)\), we will use a mutation argument to 
relate the semiorthogonal decompositions of Lemma \ref{special-decompositions}. 
For this purpose, we will make use of the following facts about mutation functors when working relatively to a base. 

\begin{lemma}
\label{lemma-mutations-relative}
Let $f \colon Y \to B$ be a morphism of smooth proper varieties, 
and let 
\begin{equation*}
\Db(Y) = \langle \cA_1, \dots, \cA_n \rangle
\end{equation*} 
be a $B$-linear semiorthogonal decomposition with admissible components. 
\begin{enumerate}
\item 
The left and right mutation functors throught any $\cA_k$ are 
$B$-linear, i.e. commute with tensoring by pullbacks of objects from $\Db(B)$. 
\item \label{exceptional-mutation-relative} If $\cA_k$ is generated by a relatively exceptional object $E$ over $B$, i.e. 
$f_* \cHom_Y(E, E) \simeq \sO_{B}$ (where $\cHom_Y(-,-)$ is the derived sheaf $\Hom$ on $Y$) and $\cA_k$ is the image of the fully faithful functor $f^*(-) \otimes E \colon \Db(B) \to \Db(Y)$, then the associated mutation functors are given by 
\begin{align*}
  \mathbf{L}_{f^*\Db(B) \otimes E} (F) & = \mathrm{Cone}((f^*f_*\cHom(E,F)) \otimes E \to F) , \\ 
  \mathbf{R}_{f^*\Db(B) \otimes E}(F) &= \mathrm{Cone}(F \to (f^*f_*\cHom(F,E))^\vee \otimes E)[-1].
\end{align*}

\item \label{serre-mutation-relative}
We have 
\begin{equation*}
\mathbf{L}_{\langle \cA_1, \dots, \cA_{n-1} \rangle}(\cA_n) = \cA_n \otimes \omega_{Y/B}
  \quad\text{and}\quad
\mathbf{R}_{\langle \cA_2, \dots, \cA_{n} \rangle}(\cA_1) = \cA_1 \otimes \omega_{Y/B}^{-1}. 
\end{equation*}
\end{enumerate}
\end{lemma}

\begin{proof}
These statements follow easily from the definitions, and in the cases of~\ref{exceptional-mutation-relative} and~\ref{serre-mutation-relative} are analogous to their absolute versions in Lemma~\ref{lemma-mutations}.  
For example, let us prove the first claim of~\ref{exceptional-mutation-relative}. 
If $\alpha \colon \cA_k \to \Db(Y)$ is the inclusion functor, then the left mutation functor is given by 
\begin{equation*}
\mathbf{L}_{\cA_k}(F) = \mathrm{Cone}(\alpha\alpha^!(F) \to F). 
\end{equation*} 
Since the right adjoint to the functor $f^*(-) \otimes E$ is $f_*(\cHom_Y(E,-)) \colon \Db(Y) \to \Db(B)$, the claimed formula for $\mathbf{L}_{f^*\Db(B) \otimes E}$ follows. 
\end{proof} 

The following result completes the proof of Theorem~\ref{main-theorem-special}. 

\begin{prop}\label{special-Ku-hat-is-CY}
There is a $\PP \widebar{V}$-linear equivalence of categories \(\widehat{\Ku}(X) \simeq \Db(W^\doubleplus)\).
\end{prop}

\begin{proof}
The following argument is
similar to \cite[Theorem 4.2]{Xie:Quadrics}. 
As in Proposition~\ref{prop-Ku-Cl}, the equivalence is obtained by comparing
the two semiorthogonal decompositions of Lemma~\ref{special-decompositions}, 
starting from 
\begin{equation*}
\Db(\widehat{X}) = \big\langle
\hat{\pi}^*\Db(\PP\widebar{V}) \otimes i'_*\sO_{Z'}(Z'),
\widehat{\Ku}(X),
\hat{\pi}^*\Db(\PP\widebar{V}) \otimes \sO_{\widehat{X}},
\hat{\pi}^*\Db(\PP\widebar{V}) \otimes \sO_{\widehat{X}}(H)
\rangle. 
\end{equation*} 

\smallskip
\noindent\textbf{Step 1.}
Mutate the first component to the far right. 
Note that by Lemma~\ref{special-blowup-facts}  we have 
$-K_{\widehat{X}} = 2H - Z' + 2h$, and that $H = h$ on the section \(\PP\mathcal{N}\) and hence also on $Z'$. 
Thus the result of the mutation is 
\begin{equation*}
\Db(\widehat{X}) = \big\langle
\widehat{\Ku}(X),
\hat{\pi}^*\Db(\PP\widebar{V}) \otimes \sO_{\widehat{X}},
\hat{\pi}^*\Db(\PP\widebar{V}) \otimes \sO_{\widehat{X}}(H), 
\hat{\pi}^*\Db(\PP\widebar{V}) \otimes i'_*\sO_{Z'}
\rangle. 
\end{equation*} 

\smallskip
\noindent\textbf{Step 2.}
Right mutate $\hat{\pi}^*\Db(\PP\widebar{V}) \otimes \sO_{\widehat{X}}(H)$ through $\hat{\pi}^*\Db(\PP\widebar{V}) \otimes i'_*\sO_{Z'}$. 
By the $\PP\widebar{V}$-linearity of the mutation functor, it suffices to understand the result for $\sO_{\widehat{X}}(H)$. 
As noted above, $H = h$ on $Z'$, and thus by Lemma~\ref{special-resolution-exceptional-divisor} we have 
\begin{equation*}
\hat{\pi}_* \cHom_{\widehat{X}}(\sO_{\widehat{X}}(H), i'_*\sO_{Z'}) 
\cong \hat{\pi}_*i'_*\sO_{Z'}(-h)
\cong \sO_{\PP \widebar{V}}(-h). 
\end{equation*} 
Using the description of the right mutation functor in Lemma~\ref{lemma-mutations-relative} and the equality $\xi = H - Z'$ from Lemma~\ref{special-blowup-facts}\ref{special-blowup-facts.canonical}, we find that the 
result of the mutation is 
\begin{equation*}
\Db(\widehat{X}) = \big\langle
\widehat{\Ku}(X),
\hat{\pi}^*\Db(\PP\widebar{V}) \otimes \sO_{\widehat{X}}, 
\hat{\pi}^*\Db(\PP\widebar{V}) \otimes i'_*\sO_{Z'}, 
\hat{\pi}^*\Db(\PP\widebar{V}) \otimes \sO_{\widehat{X}}(\xi)
\rangle. 
\end{equation*} 

\smallskip
\noindent\textbf{Step 3.}
Left mutate $\hat{\pi}^*\Db(\PP\widebar{V}) \otimes i'_*\sO_{Z'}$ through $\hat{\pi}^*\Db(\PP\widebar{V}) \otimes \sO_{\widehat{X}}$. 
Similarly to the previous step, the result is 
\begin{equation*}
\Db(\widehat{X}) = \big\langle
\widehat{\Ku}(X),
\hat{\pi}^*\Db(\PP\widebar{V}) \otimes \sO_{\widehat{X}}(-Z'), 
\hat{\pi}^*\Db(\PP\widebar{V}) \otimes \sO_{\widehat{X}}, 
\hat{\pi}^*\Db(\PP\widebar{V}) \otimes \sO_{\widehat{X}}(\xi)
\rangle. 
\end{equation*} 

\smallskip
\noindent\textbf{Step 4.}
Mutate $\widehat{\Ku}(X)$ through $\hat{\pi}^*\Db(\PP\widebar{V}) \otimes \sO_{\widehat{X}}(-Z')$ and then mutate $\hat{\pi}^*\Db(\PP\widebar{V}) \otimes \sO_{\widehat{X}}(-Z')$ to the far right. 
Using again the formulas $-K_{\widehat{X}} = 2H - Z' + 2h$ and $\xi = H - Z'$, the result is 
\begin{equation*}
\begin{multlined}
\Db(\widehat{X}) = \big\langle
\mathbf{R}_{\hat{\pi}^*\Db(\PP\widebar{V}) \otimes \sO_{\widehat{X}}(-Z')}\widehat{\Ku}(X), 
\hat{\pi}^*\Db(\PP\widebar{V}) \otimes \sO_{\widehat{X}}, \\ 
\hat{\pi}^*\Db(\PP\widebar{V}) \otimes \sO_{\widehat{X}}(\xi), 
\hat{\pi}^*\Db(\PP\widebar{V}) \otimes \sO_{\widehat{X}}(2\xi) 
\rangle. 
\end{multlined}
\end{equation*}
Comparing with the second semiorthogonal decomposition from
Lemma~\ref{special-decompositions}, this shows that
\begin{equation*}
\mathbf{R}_{\hat{\pi}^*\Db(\PP\widebar{V}) \otimes \sO_{\widehat{X}}(-Z')}\widehat{\Ku}(X)
= \hat{a}_{\widehat{X}}^*j_*\Db(W^\doubleplus) \otimes \sO_{\widehat{X}}(-\Lambda).
\end{equation*}
Since all of the functors $\mathbf{R}_{\hat{\pi}^*\Db(\PP\widebar{V}) \otimes \sO_{\widehat{X}}(-Z')}$, $\hat{a}_{\widehat{X}}^*$, and $j_*$ are $\PP \widebar{V}$-linear, it follows that there is a $\PP \widebar{V}$-linear equivalence $\widehat{\Ku}(X) \simeq \Db(W^\doubleplus)$; 
explicitly, the composition 
\begin{equation*}
\mathbf{L}_{\hat{\pi}^*\Db(\PP\widebar{V}) \otimes \sO_{\widehat{X}}(-Z')}
\circ \big(- \otimes \sO_{\widehat{X}}(-\Lambda)\big)
\circ \hat{a}_{\widehat{X}}^*
\circ j_*
\colon \Db(W^\doubleplus) \to \widehat{\Ku}(X)
\end{equation*} 
is an equivalence.
\end{proof}

\bibliographystyle{amsalpha}
\bibliography{main}
\end{document}